\documentclass[reqno,11pt]{amsart}
\numberwithin{equation}{section}
\usepackage{amsmath}
\usepackage{esint} 
\usepackage{amsthm}
\usepackage{epsfig}
\usepackage{psfrag}
\usepackage{graphicx}
\usepackage{graphpap,latexsym,epsf}
\usepackage{color}
\usepackage{amssymb,mathrsfs,enumerate}
\definecolor{citegreen}{rgb}{0.1,0.4,0.1}
\definecolor{refred}{rgb}{0.5,0,0}
\usepackage[colorlinks, citecolor=citegreen, linkcolor=refred]{hyperref}
\newcommand{\R}{\mathbb{R}}

\newcommand{\Sph}{\mathbb{S}}
\def\DDD{{\rm D}}

\def\HHH{{\rm H}}
\def\RRR{{\mathrm R}}

\def\a{\alpha}

\newcommand{\pa}{\partial}
\newcommand{\Om}{\Omega}
\newcommand{\ffi}{\varphi}

\newcommand{\ep}{\varepsilon}
\newcommand{\rmd}{{\rm d}}

\newcommand{\Ric}{{\rm Ric}}

\newcommand{\D}{{\rm D}}
\newcommand{\DD}{{\rm D}^2}
\newcommand{\De}{\Delta}
\newcommand{\ho}{{\rm h}}

\newcommand{\Ho}{{\rm H}}

\newcommand{\g}{g}

\newcommand{\Ricg}{{\rm Ric}_g}

\newcommand{\na}{\nabla}
\newcommand{\nana}{\nabla\nabla}
\newcommand{\Deg}{\Delta_g}

\newcommand{\chg}{{\rm h}^{(g)}}
\newcommand{\Hg}{{\rm H}_g}

\newcommand{\hhh}{{\rm h}}

\mathchardef\emptyset="001F

\definecolor{vgreen}{rgb}{0.1,0.5,0.2}
\definecolor{viola}{RGB}{85,26,139}

\newtheorem{theorem}{Theorem}[section]
\newtheorem{remark}{Remark}

\newtheorem{corollary}[theorem]{Corollary}

\newtheorem{proposition}[theorem]{Proposition}

\newtheorem{lemma}[theorem]{Lemma}

\begin{document}

\hyphenation{ma-ni-fold}

\title[Monotonicity formulas in potential theory]
{Monotonicity formulas in potential theory}

\author[V.~Agostiniani]{Virginia Agostiniani}
\address{V.~Agostiniani, Universit\`a degli Studi di Verona,
Strada le Grazie 15, 37134 Verona, Italy}
\email{virginia.agostiniani@univr.it}

\author[L.~Mazzieri]{Lorenzo Mazzieri}
\address{L.~Mazzieri, Universit\`a degli Studi di Trento,
via Sommarive 14, 38123 Povo (TN), Italy}
\email{lorenzo.mazzieri@unitn.it}


\begin{abstract} 
Using the electrostatic potential $u$ due to a uniformly charged body $\Om \subset \R^n$, $n\geq 3$, we introduce a family of 
monotone quantities associated with the level set flow of $u$. The derived monotonicity formulas are exploited to deduce a new quantitative version of the classical Willmore inequality.
\end{abstract}

\maketitle


\smallskip
\noindent 

\date{\today}

\maketitle


\section{Introduction}
\label{sec:introd}


\subsection{Setting of the problem and statement of the main result.}

We consider the electrostatic potential due to a charged body, modelled by a bounded domain $\Omega$ with $\mathscr{C}^{2, \a}$-boundary,
for some $0<\alpha<1$.
The potential is defined as the unique solution $u$ 
of the following problem in the exterior domain
\begin{equation}
\label{eq:pb}
\left\{
\begin{array}{rcll}
\displaystyle
\De u\!\!\!\!&=&\!\!\!\!0 & {\rm in }\quad\R^n\setminus\overline\Om \,,\\
\displaystyle
  u\!\!\!\!&=&\!\!\!\!1  &{\rm on }\ \ \pa\Om \, ,\\
\displaystyle
u(x)\!\!\!\!&\to&\!\!\!\!0 & \mbox{as }\ |x|\to\infty \, .
\end{array}
\right.
\end{equation}
Throughout the paper we assume
that $\pa\Om$ is a regular level set of $u$.
It is worth pointing out that for every $0<t\leq1$ the level set $\{u=t\}$ is compact, due to the properness of $u$. 
Moreover, we have that for every $t>0$ sufficiently small the level set $\{u=t\}$ is diffeomorphic to a 
$(n-1)$-dimensional sphere, and thus connected. These properties can be deduced from expansion~\eqref{eq:u_exp} below.

A natural quantity associated with a solution to problem~\eqref{eq:pb} is the electrostatic capacity of $\Om$, which is defined as
\begin{equation*}
\label{def:cap}
{\rm Cap} (\Omega) \, 
:= 
 \,  \inf  \, \bigg\{  \frac{1}{(n-2) |\Sph^{n-1}|}\int_{ \R^n}  |\D w|^2 \rmd \mu \,\, \bigg| \,\,  w\in\mathscr C_{c}^{\infty}(\R^n),\,w\equiv1\mbox{ in }\Om 
\bigg\} \, .
\end{equation*}
We recall that the capacity of $\Om$ can be computed in 
terms of the electrostatic potential $u$ as
\begin{equation}
\label{formula_cap}
 {\rm Cap} (\Omega) \,(n-2)\, |\Sph^{n-1}| \, = \int\limits_{  \pa \Om }  |\D u| \, \rmd \sigma \,\, = \!\! \int\limits_{ \{ u = t \}} \!\!\! |\D u| \, \rmd \sigma  \, ,
\end{equation}
where in the second equality we have used the Divergence Theorem 
and the equation $ \Delta u =0$. 
This last fact can be rephrased by saying that the function $F_0:  [1, + \infty) \, \longrightarrow \, \R $, given by
\begin{equation}
\label{eq:f0}
\tau \,\, \longmapsto \,\, F_0(\tau)  
\,=  \!\!\!\!\!
  \int\limits_{ \{ u \,= \, 1/\tau \}} \!\!\!\!\!\!  
|\D u| \, \rmd \sigma \,,
\end{equation}
is constant and 
$F_0 (\tau) \equiv {\rm Cap} (\Omega)(n-2)|\Sph^{n-1}|$,
for every $t \in (0, 1]$. In analogy with~\eqref{eq:f0}, we introduce for 
$\beta\geq0$ the functions $F_\beta:  [1, + \infty) \, \longrightarrow \, \R $ given by
\begin{equation}
\label{eq:fb}
\tau \,\, \longmapsto \,\, F_\beta(\tau) \, 
\,= 
\,\,\, \tau^
{\,\beta\,(\frac{n-1}{n-2})}
\!\!\!\!\! \int\limits_{ \{ u \,= \, 1/\tau \}} \!\!\!\!  
|\D u|^{\beta+1} \, \rmd \sigma \,.
\end{equation}
To describe a first property of the functions $F_\beta$'s,
observe that, using the following classical asymptotic expansions 
at infinity (see for instance~\cite{Kellogg})
\begin{align}
u \,\, &= \,\,{\rm Cap}(\Omega)\,\, |x|^{2-n}  \, + \, o(|x|^{2-n}) \, ,
\nonumber\\
{\rm D}_{\!\mu}u 
\,\, &=\,- (n-2) \,  {\rm Cap}(\Omega)\,\, |x|^{-n}\,x_{\mu} \,+ \, o(|x|^{1-n}) \, ,
\label{eq:u_exp}\\
\D_\mu\D_\nu u 
\,\, &= \,\,(n-2) \, {\rm Cap}(\Omega)\,\, |x|^{-n - 2}
\left(n \, {x_{\mu}x_{\nu}}
- {|x|^2}\delta_{\mu\nu}\right) \, + \, o(|x|^{-n}) \, ,
\nonumber
\end{align}
one can easily compute the limit
\begin{equation}
\label{eq:lim_fb}
\lim_{\tau\to +\infty} \, F_\beta(\tau) 
\,\, = \,\,
\left[\, {{\rm Cap}(\Omega)} \, \right]^{\frac{n-2-\beta}{n-2}} 
\!(n-2)^{\beta+1}\, |\Sph^{n-1}| \, .\phantom{\qquad}
\end{equation}

Before proceeding with the statement of the main result, it is worth describing some other features of the functions $\tau \mapsto F_\beta(\tau)$.
First of all, we notice that such functions are well defined, since the integrands are globally bounded and the level sets of $u$ have finite hypersurface area. In fact, since $u$ is harmonic, the level sets of $u$ have locally finite 
$(n-1)$-dimensional Haurdorff measure $\mathscr{H}^{n-1}$
(see for instance~\cite{Hardt_Simon} and 
the references therein). Moreover, by the properness of $u$, 
such level sets are compact and thus with finite hypersurface area. To describe another important feature of the $F_\beta$'s, we recall that in the case where $\Om$ is a ball the explicit solution to problem~\eqref{eq:pb} is given by (a multiple of) the Green function. Hence, the expansions~\eqref{eq:u_exp} deprived of the reminder terms yield in this case explicit formul{\ae} for $u$, $\DDD u$ and $\D\D u$ on the whole $\R^n \setminus \overline{\Om}$. Tacking advantage of this observation, one easily realizes that the quantities
\begin{equation}
\label{charact_funct}
\R^n \setminus \overline{\Omega}\, \ni \,x \, \longmapsto \,\,   \frac{|\DDD u|}{u^{\!\frac{n-1}{n-2}}} \,(x) \qquad \hbox{and} \qquad  [1, +\infty) \, \ni \, \tau \, \longmapsto \!\!\!\!\int\limits_{\{u = 1/\tau\}}\!\!\!\!\!
  u^{\!\frac{n-1}{n-2}}
 \,\rmd\sigma 
\end{equation}
are constant on rotationally symmetric solutions. Also notice {\em en passant} that the square of the first quantity is known in the literature as the $P$-function naturally associated with problem \eqref{eq:pb},
see for instance 
\cite{Cra_Fra_Gaz,Enc_Per,Fra_Gaz_2008,Fra_Gaz_Kaw,Gar_Sar,Pay_Phi_1,Pay_Phi_2}. 
In Subsection \ref{sub:strategy} we will shade some lights on the 
geometric nature of such a function. 
We also note that, for every 
$\beta\geq 0$, the functions $\tau \mapsto F_\beta(\tau)$
can be rewritten in terms of the quantities appearing in~\eqref{charact_funct} as
\begin{equation}
\label{eq:fp2}
{F}_\beta(\tau)\,\,\,=\!\!\!
\int\limits_{\{u = 1/\tau \}}\!\!\!\!
\left(   \frac{|\DDD u|}{u^{\!\frac{n-1}{n-2}}}\right)
^{\!\!\beta
+1} \!\! u^{\frac{n-1}{n-2}}
\,\, \rmd\sigma .
\end{equation}
Therefore they are constant on rotationally symmetric solutions. In contrast with this, our main result states that the functions 
$\tau \mapsto F_\beta(\tau)$ 
are in general monotonically nondecreasing and that the monotonicity is strict unless $u$ is rotationally symmetric.
\begin{theorem}[Monotonicity-Rigidity Theorem]
\label{thm:main}
Let $u$ be a solution to problem~\eqref{eq:pb} and let 
$F_\beta : [1,+\infty) \rightarrow \R$ 
be the function defined  in~\eqref{eq:fb}. Then, the following properties hold true.
\begin{itemize}
\smallskip
\item[(i)] \underline{Continuity.} For every 
$\beta\geq0$, 
the function $F_\beta$ 
is continuous.  
\smallskip
\item[(ii)] \underline{Differentiability, Monotonicity \& Rigidity.} For every 
$\beta\geq(n-2)/(n-1)$, 
the function $F_\beta$ is continuously differentiable and the derivative admits for every $\tau \geq 1$ the integral representation
\begin{align}
\label{eq:der_fb}
\nonumber
\phantom{\qquad\,\,} { F}_\beta'(\tau)
\,\,=\,\, - \,\, \beta 
\!\!\!\!\! \int\limits_{ \{ u \,< \, 1/\tau \}} \!\!\!\!\!\!  u^{2-\beta\,(\frac{n-1}{n-2}) } \, |\DDD u|^{\beta-2} \, &\bigg\{ \, \left[\,\, 
\big|\D\D u\big|^2-{ \big(\textstyle\frac n{n-1}}\big)
\big|\D|\D u|\big|^2 \,\, \right]  \,\, + \\
\phantom{\int\limits_{ \{ u \,= \, 1/\tau \}}} & + \,\,
\big({\textstyle\beta-\frac{n-2}{n-1}}\big)
\,\, \big|\D^{T}|\D u|\big|^2 \,\,+ \\
& \nonumber
+ \,\,
\big({\textstyle\beta-\frac{n-2}{n-1}}\big) \,\,
|\D u|^2
\left[ \, 
{\rm H}- 
\big({\textstyle\frac{n-1}{n-2}}\big) |\D \log u|
 \, \right]^2 
 \,\bigg\}
\,\, \rmd\mu \, ,
\end{align}
where $\HHH (x)$ is the mean curvature of the level set of $u$ passing through $x$, computed with respect to the unit normal vector field
$\nu = - \D u / |\D u|$, and $\D^T$ denotes the tangential part of the gradient.
In particular, the derivative is always nonpositive. Furthermore, the sign of the derivative is strict for all $\tau \in [1,+\infty)$, unless the function $F_\beta$ is constant and $u$ is rotationally symmetric.

\smallskip

\item[(iii)] \underline{Convexity \& Rigidity.} For every 
$\beta\geq(n-2)/(n-1)$, 
the function $F_\beta$ is convex and the convexity is strict, unless the function is constant and $u$ is rotationally symmetric. Moreover, for every $\tau\in[1,+\infty)$ where $F_\beta$ is twice differentiable, the second derivative obeys the formula 
\begin{align}
\label{eq:der2_fb}
\nonumber
{ F_\beta''}(\tau)
\,\,=\,\,  \beta \,\,\,\tau^
{\,\beta\,(\frac{n-1}{n-2})-4}
\!\!\!\!\!\!  \int\limits_{ \{ u \,= \, 1/\tau \}} \!\!\!\!\!\!   |\DDD u|^{\beta-1} \, &\bigg\{ \, \left[\,\, 
| {\hhh} |^2 - \, { \big(\tfrac {1}{n-1}\big)} \, \HHH^2
\,\, \right]  \,\, + \\
\phantom{\int\limits_{ \{ u \,= \, 1/\tau \}}} & + \,\,
\beta 
\,\, \left|\frac{\D^{T}|\D u|}{|\D u |}\right|^2 \,\,+ \\
& \nonumber
+ \,\,
\big({\textstyle\beta-\frac{n-2}{n-1}}\big) \,
\left[ \, 
{\rm H}- 
\big({\textstyle\frac{n-1}{n-2}}\big) |\D \log u|
 \, \right]^2 
 \,\bigg\}
\,\, \rmd\sigma \, ,
\end{align}
where $\hhh$ is the second fundamental form of the level set $\{u=1/\tau\}$, computed with respect to the unit normal vector field
$\nu = - \D u / |\D u|$, 
and $\D^T$ denotes the tangential part of the gradient. In particular, the second derivative is always nonnegative, wherever defined. Furthermore, the sign is strict, unless the function $F_\beta$ is constant and $u$ is rotationally symmetric.
\end{itemize}
\end{theorem}

It is worth mentioning that monotonicity formulas 
are nowadays known to play a fundamental
role in geometric analysis
(dropping any attempt of being complete, we mention \cite{huisken1990,Hui_Ilm,Perel_1} 
and also the more recent \cite{Brendle}),
as well as in the study of geometric properties
of harmonic functions on manifolds subject to 
suitable curvature lower bounds
\cite{Colding_1,Colding_Minicozzi_2,Colding_Minicozzi}.
With regard to the last mentioned references,
we postpone further comments to the {\em Added note}
at the end of this section.

Before stating in Subsection \ref{sub:geom} the main geometric implications of the above theorem, let us list some remarks, in which some technical details concerning our monotonicity formulas are taken into account. 

\begin{remark}
\label{rem:deretta}
Let us observe that if $1/\tau$ is a regular value of $u$, then the function $F_\beta$ is differentiable at $\tau$ for every $\beta \geq 0$, and a direct computation gives 
\begin{equation}
\label{eq:der_up}
\phantom{\qquad}{F_\beta'}(\tau)
\,=\, - \, \beta \,\,\,\tau^
{\,\beta\,(\frac{n-1}{n-2})-2}
\!\!\!\!\!\! \int\limits_{ \{ u \,= \, 1/\tau \}} \!\!\!\!  |\DDD u|^\beta
\left[\, \HHH-
\big(\tfrac{n-1}{n-2}\big)
\, {|\DDD \log u|}
\, \right] 
 \, \rmd \sigma \, ,
\end{equation}
where $\HHH$ is the mean curvature of the level set $\{u=1/\tau\}$ computed with respect to the unit normal vector field
$\nu = - \D u / |\D u|$. 
In Section~\ref{consequences} we will combine the above formula~\eqref{eq:der_up} with the sign coming from the Monotonicity Formula~\eqref{eq:der_fb} for $\beta \geq (n-2)/(n-1)$, in order to draw several geometric conclusions. 
For example, such combination implies at once 
the non existence of minimal level sets of $u$,
and in turn the non existence of 
smooth minimal compact hypersurfaces in $\R^n$.
\end{remark}

\begin{remark}
\label{rem:uno_u}
Notice that for $\beta \geq 1$ formula~\eqref{eq:der_up} is well-posed also in the case where $\{u = 1/\tau\}$ is not a regular level set of $u$. 
Indeed, as already observed, it is well-known that, 
since $u$ is harmonic and proper,
the $\mathscr H^{n-1}$-measure of each of 
its level sets is finite. 
Moreover, by the results in~\cite{Nadirashvili} 
(see also~\cite{Caf_Fri} and~\cite{Che_Nab_Val}), 
the Hausdorff dimension of its critical set 
is bounded above by $(n-2)$. 
In particular, the unit normal is well defined 
$\mathscr{H}^{n-1}$-almost everywhere on each level set 
and so does the mean curvature $\HHH$. 
In turn, the integrand in~\eqref{eq:der_up} is well defined $\mathscr{H}^{n-1}$-almost everywhere. 
Finally, we observe that where $|\D u| \neq 0$ it holds
\begin{equation*}
|\D u|^{\beta}\, \HHH \,\, = \,\,\, |\D u|^{\beta-3} \,\DD u (\D u , \D u) \, .
\end{equation*}
Since $|\DD u|$ is uniformly bounded 
in $\R^n\setminus\overline\Om$, 
this shows that the integrand in~\eqref{eq:der_up} is essentially bounded and thus summable on every level set of $u$, provided $\beta \geq 1$. In the case where $(n-2)/(n-1)\leq \beta<1$, it is no longer possible to infer that the function $|\D u|^{\beta} \, \HHH$ is essentially bounded on the critical level sets of $u$. 
However, we will prove in Corollary~\ref{cor:diffi} -- which is at the core of Theorem~\ref{thm:main}-(ii) -- that the right hand side of~\eqref{eq:der_up} admits a unique continuous extension to the critical values of $u$. In this sense, the function $|\D u|^{\beta} \, \HHH$ can be understood as an integrable function also for small admissible values of $\beta$.
\end{remark}

\begin{remark}
\label{rem:seconder}
Formula~\eqref{eq:der2_fb}
for the second derivative of $F_\beta$ 
follows at once from~\eqref{eq:der_fb}, through the Coarea Formula and the pointwise identity
\begin{equation*}
\big|\D\D u\big|^2 - \, { \big(\textstyle\frac n{n-1}}\big) \, 
\big|\D|\D u|\big|^2  
 - \, \big(\tfrac{n-2}{n-1}\big)
\, \big|\D^{T}|\D u|\big|^2 \, = \,\,\,|\D u|^2 \, \left[ | {\hhh} |^2 - \, { \big(\tfrac {1}{n-1}\big)} \, \HHH^2 \right]\, .
\end{equation*}
However, at a regular value of $u$, it can also be deduced from~\eqref{eq:der_up} 
via a direct computation. 
\end{remark}


\subsection{Geometric implications}
\label{sub:geom}


In this subsection, we describe the main geometric conclusions that can be drawn from the Monotonicity-Rigidity Theorem~\ref{thm:main}, deferring the proofs to Section~\ref{consequences}.

A nowadays classical statement about the geometry of smooth closed surfaces in the Euclidean three-dimensional space is the so called Willmore Inequality, namely
\begin{equation}
\label{eq:will2}
16 \pi \,\, \leq \,    \int_{ \pa \Omega} \!\! \HHH^2  \, \rmd \sigma \, ,
\end{equation}
where the equality is achieved if and only if the domain $\Omega$ 
is a round ball. Its validity, together with its extension to higher dimensions, has been established by the joint efforts of several authors (see \cite[Theorem 3]{Chen_2}, \cite{Chen_1}, and also \cite{Willmore}, together with the references therein). Applying the Monotonictiy-Rigidity Theorem with $\beta = n-2$, one gets immediately that $F_{n-1}(1) \geq \lim_{\tau \to + \infty} F_{n-1} (\tau)$, where the limit equals
$(n-2)^{n-1} |\Sph^{n-1}|$, in view of formula~\eqref{eq:lim_fb}. On the other hand, it is not hard to use the sign of~\eqref{eq:der_up} in combination with the H\"older inequality  in order to obtain an upper bound for $F_{n-1}(1)$ in terms of the $L^{n-1}$-norm of the mean curvature of $\pa \Omega$. As a consequence, one recovers a new proof of the $(n-1)$-dimensional Willmore Inequality 
\begin{equation}
|\Sph^{n-1}| \,\, \leq \,   \int_{ \pa \Omega}  \left|\frac{\HHH}{n-1}\right|^{n-1}  \rmd \sigma \, ,
\end{equation}
together with its corresponding rigidity statement, which characterises the round balls. Using the optimal threshold for the exponent $\beta$ in our Monotonicity Formulas, we also deduce a novel sharp quantitative version of the above inequality.
\begin{theorem}
[Quantitative Willmore-type inequality]
\label{cor:will-def} 
Let $\Omega \subseteq \R^n$, 
$n\geq 3$, 
be a bounded domain with smooth boundary. Then, the inequality	
\begin{equation}
\label{eq:will}
 \Big|  {F'_{\frac{n-2}{n-1}}}(1) \Big| 
 \,\, \leq \,\,  A(n) \, \left[{\rm Cap}(\Omega) \right]^{(n-2)/(n-1)} \left[  \left( \int_{ \pa \Omega}  \left|\frac{\HHH}{n-1}\right|^{n-1} \!\!\!\! \rmd \sigma \right)^{\!\!1/(n-1)} \!\!\!\!\!\!\!\!\!\!\!\! - \,\,\,\,\, \left|\Sph^{n-1}\right|^{\,1/(n-1)} \right]
\end{equation}
holds true, where $\HHH$ is the mean curvature of $\pa \Omega$ and the positive constant $A(n)$ is explicitly given by
\begin{equation}
\label{A(n)}
A(n) \,\, = \,\, { \, (n-2)^{(2n-3)/(n-1)} \, \left|\Sph^{n-1}\right|^{(n-2)/(n-1)}} \, .
\end{equation}
Moreover, the deficit on the left hand side is optimal in the sense that if ti vanishes, then the right hand side also vanishes and $\Omega$ is a round ball.
\end{theorem}

\begin{remark}
The optimal deficit on the left hand side of~\eqref{eq:will} can be written more explicitly with the help of~\eqref{eq:der_fb}, so that the Quantitative Willmore-type Inequality rewrites as 
\begin{equation*}
 \int\limits_{\R^n \setminus \overline{\Omega}}\!\! u   \left[  \frac{ 
\big|\D\D u\big|^2-{ \big(\textstyle\frac n{n-1}}\big)
\big|\D|\D u|\big|^2 }{|\DDD u|^{n/(n-1)}}\right] \,\, \rmd \mu
 \, \leq \,\,  \hat{A}(n) 
\left[{\rm Cap}(\Omega) \right]^{\frac{n-2}{n-1}} \left[  \left( \int_{ \pa \Omega}  \left|\frac{\HHH}{n-1}\right|^{n-1} \!\!\!\! \rmd \sigma \right)^{\!\!\frac{1}{n-1}}
\!\!\!\!\! - \,\,\, \left|\Sph^{n-1}\right|^{\frac{1}{n-1}}\right]
\end{equation*}
where $u$ is the capacitary potential of $\Omega$ and the positive constant $\hat{A}(n)$ is explicitly given by
\begin{equation*}
\hat{A}(n) \, = \, {(n-1) \, \left[ \, (n-2) \, \left|\Sph^{n-1}\right| \,\,  \right]^{(n-2)/(n-1)}} \, .
\end{equation*}
\end{remark}

In~\cite{Ago_Fog_Maz} we extend the validity of the Willmore-type Inequality to the case of manifolds with nonnegative Ricci curvature and Euclidean volume growth, obtaining among the consequences some new characterizations of the Asymptotic Volume Ratio. This context is also
the more natural for drawing a detailed comparison 
with the work of Colding and Minicozzi,
as explained in the \emph{Added note}.

Another important instrument in the study of the geometry of the hypersurfaces in $\R^n$ is the so called Minkowski inequality. It says that if $\Omega$ is convex, then
\begin{equation}
\label{eq:mink}
|\pa \Omega|^{(n-2)/({n-1})} \, |\Sph^{n-1}|^{1/(n-1)} \,\, \leq \,   \int_{ \pa \Omega}  \frac{\HHH}{n-1}  \,\, \rmd \sigma \, .
\end{equation}
A new proof of this classical result is provided in~\cite{Fog_Maz_Pin}, 
using the level set flow of $p$-capacitary potentials.
Taking advantage of the beautiful results of~\cite{Ger} and~\cite{Urb}
about the long time behaviour of the Inverse Mean Curvature Flow for domains 
that are mean convex and starshaped,
it is possible to extend the validity of 
inequality~\eqref{eq:mink}
to this class of domains 
(see~\cite{Guan_Li}).
Another route, still based on the Inverse Mean Curvature Flow and 
suggested by Huisken in~\cite{Hui_video},
is extending the Minkowski inequality to the case of 
outward minimizing domains. These are also mean convex, but not necessarily diffeomorphic to spheres. 
Furthermore, in~\cite{Chang_Wang_2011,Chang_Wang_2013},
the Minkowski inequality is proved for mean convex domains 
whose boundary has positive scalar curvature,
using techniques from optimal transport. 
An open question is whether the mean convexity alone is sufficient to imply the validity of~\eqref{eq:mink} or not. 
We will address this question in
a forthcoming work. 
Here, we provide a weighted version of inequality~\eqref{eq:mink} that has the advantage of not requiring any geometric restriction on $\Omega$.
\begin{theorem}
[Weighted Minkowski Inequality]
\label{cor:mink-def} 
Let $\Omega \subseteq \R^n$, 
$n\geq 3$, 
be a bounded domain with smooth boundary. Then, the inequality
\begin{equation}
\label{eq:minki}
\Big|F'_{\frac{n-2}{n-1}}(1)\Big|  
\,\,\leq \,\,   
A(n)
\left[
\frac{{\rm Cap}(\Omega)}
{|\pa\Omega|} 
\right]^{\frac{n-2}{n-1}}
\Bigg[
\int_{ \pa \Omega}  \frac{\HHH}{n-1} \,\, \rmd \overline\sigma  
 \,\,\,- \,\,\, 
 |\pa\Omega|^{(n-2)/({n-1})} \,\left|\Sph^{n-1}\right|^{\,1/(n-1)}
 \Bigg]
\end{equation}
holds true, where $\HHH$ is the mean curvature of $\pa \Omega$, the measure element $\rmd \overline{\sigma}$ is defined as
\begin{equation*}
\rmd \overline\sigma  \,\,\,\, = \,\,\,\, \left( \frac{|\D u|}{\fint_{\pa \Omega}\!
  |\DDD u|
\,\, \rmd\sigma } \right)^{\!(n-2)/(n-1)}  \rmd \sigma \,\, ,
\end{equation*}
and the positive constant $A(n)$
is given by formula~\eqref{A(n)}.
Moreover, the deficit on the left hand side is optimal in the sense that if it vanishes, then the right hand side also vanishes and $\Omega$ is a round ball.
\end{theorem}
Concerning the weighted measure that appears in the above statement, it is intriguing to observe that by the reverse Jensen's inequality, one has that
\begin{equation*}
\int\limits_{\pa \Omega} \rmd \overline\sigma 
\,\, \leq \,\, 
\left|  \pa \Omega\right| \, ,
\end{equation*}
so that the mass of $\pa \Omega$ with respect to $\overline\sigma$ is less than or equal to the usual one.


\subsection{Strategy of the proof. } 
\label{sub:strategy}


In this subsection, we present the main ideas underlying the proof
of the Monotonicity-Rigidity Theorem.
To do this, we focus for simplicity on the case $\beta =2$. 
Our strategy consists of two main steps. 
The first step is the construction of a {\em cylindrical ansatz},
that is a metric $g$ conformally equivalent to the 
Euclidean metric $g_{\R^n}$
through the conformal factor 
$u^{\frac2{n-2}}$, namely
\begin{equation*}
g \, = \, u^{\frac2{n-2}} \,g_{\R^n} \,.
\end{equation*}
The reason for the name is that when $u$ is rotationally symmetric,
then $g$ is the cylindrical metric. 
Before proceeding, we recall that the same strategy described here
is at the basis of the results of~\cite{Ago_Maz_2}
and of~\cite{Bor_Maz_1,Bor_Maz_2}, 
where \emph{static metrics} and the associated \emph{static potentials} are considered 
in place of the Euclidean metric and 
the corresponding electrostatic potential.
The cylindrical ansatz leads to a reformulation of 
problem~\eqref{eq:pb} where the new metric $g$ and the 
$g$-harmonic function $\ffi=-\log u$ fulfill the 
quasi-Einstein type equation
\begin{equation*}
\Ricg \, - \, \nana\ffi \, + \, \frac{d\ffi\otimes d\ffi}{n-2}
\,\, = \,\,
\frac{|\na\ffi|^2_g}{n-2}\,\g \,,  \quad\qquad \hbox{in \ $M$}.
\end{equation*}
Here, $M=\R^n\setminus\Om$ and 
$\na$ is the Levi-Civita connection of $g$.
Before proceeding, it is worth pointing out that taking 
the trace of the above equation gives
\begin{equation*}
\frac{\RRR_g}{n-1} \,= \, \frac{|\na \ffi|_g^2}{n-2} \, ,
\end{equation*}
where $\RRR_g$ is the scalar curvature of the conformal metric $g$. 
Hence, $|\na \ffi|_g^2$ is expected to be constant
precisely when $(M,g)$ is isometric to a round cylinder.
Noticing that
\begin{equation}
\label{eq:P_function}
|\na\ffi|_\g 
\,\, = \,\,  \frac{|\DDD u|}{u^{\!\frac{n-1}{n-2}}}
\end{equation}
and recalling the little discussion after formula~\eqref{charact_funct}, we obtain a clear geometric interpretation of the constancy of the $P$-function 
$x\mapsto P(x)  =  \left( |{\rm D}u|^2/u^{2(n-1)/(n-2)} \right)(x)$,
which is naturally
associated with problem~\eqref{eq:pb}.

The second step of our strategy consists in proving via a splitting principle that the metric $g$ has indeed a product structure, provided the hypothesis of the 
Rigidity statement is satisfied
(splitting techniques have been successfully employed
in the context of partial differential equations for 
example in \cite{Ago_Maz_1,Farina}).
More in general, 
we use the above conformal reformulation of the original system combined with the Bochner identity to deduce the equation
\begin{equation*}
\Delta_\g|\na \ffi|_\g^{2}\, - \, \big\langle\na|\na \ffi|^{2}_\g \, \big| \,\na\ffi
\big\rangle_{\!\!\g} \, = \,  2  \,  
  \big|\nana \ffi\big|_\g^2
  \,.
\end{equation*}
Observing that the drifted Laplacian appearing on the left hand side is formally self-adjoint with respect to the weighted measure 
${\rm e}^{-\ffi}\rmd \mu_g$, 
we integrate by parts and obtain, 
for every $s\geq0$, the integral identity
\begin{equation}
\label{caso_p_2}
\int\limits_{\{\ffi=s\}} \!\!
|\na \ffi|_\g^{2}\,\,\HHH_g 
\,\,\rmd\sigma_{\!g}  
\,\,\, =\,\,\,
{\rm e}^s\!\!\!\int\limits_{\{\ffi> s\}} \!\!
\frac{   
\big|\nana \ffi\big|_\g^2
   }{{\rm e}^{\ffi}}
\,\,\rmd\mu_g \, ,
\end{equation}
where $\Hg$ is the mean curvature of the level set $\{ \ffi = s \}$ inside the ambient $(M,g)$, computed with respect to the normal vector field $\na \ffi/ |\na\ffi|_g$. To give an effective interpretation of the above identity, it is now convenient to consider, for every $\beta \geq 0$,
the function
$\Phi_\beta: [0, +\infty) \longrightarrow \R$, given by
\begin{equation*}
\Phi_\beta(s) 
\,\,=\!\!\!
\int\limits_{\{\ffi = s\}}\!\!\!
|\na \ffi|_g^{\beta+1} \,\,\rmd \sigma_{\!g} \, .
\end{equation*}
A direct computations shows that, for $\beta=2$, one has
\begin{equation}
\label{eq:der_b_2}
{\Phi_2'} (s) \, = \, - \, 2  \!\!\! \int\limits_{\{\ffi=s\}} \!\!
|\na \ffi|_\g^{2}\,\,\HHH_g 
\,\,\rmd\sigma_{\!g}  
\,\,\, =\,\,\, - \, 2 \,\,
{\rm e}^s\!\!\!\int\limits_{\{\ffi> s\}} \!\!
\frac{   
\big|\nana \ffi\big|_\g^2
   }{{\rm e}^{\ffi}}
\,\,\rmd\mu_g \,\, \leq \,\, 0\,  ,
\end{equation}
so that the function $s \mapsto \Phi_2(s)$ is monotone.
Also, under the hypothesis of the Rigidity statement, the left hand side of the above identity vanishes at some point and thus the Hessian of $\ffi$ must be zero in an open region of $M$. In turn, by analyticity, it vanishes everywhere. On the other hand, the asymptotic behavior of $u$ implies that $\ffi(x)\to + \infty$ when $x\to \infty$.
In particular, $\na \ffi$ is a nontrivial parallel vector field. Hence, it provides a natural splitting direction for the metric $g$,
which can then be proved to have a product structure.
Finally, using the fact that $g_{\R^n}$ is flat and thus $g$ is conformally flat by construction, we can argue as in \cite{Ago_Maz_1} that the cross sections of the Riemannian product $(M,g)$ are indeed metric spheres and that in turn $(M,g)$ is isometric to a round cylinder.

For an arbitrary $\beta \geq (n-2)/(n-1)$, we obtain in place of~\eqref{caso_p_2} and~\eqref{eq:der_b_2}
the more general sequel of identities
\begin{eqnarray}
\Phi_\beta' (s) & = & - \, \beta  \!\!\!\int\limits_{\{\ffi=s\}} \!\!
{   |\na \ffi|_\g^\beta\,\HHH_g }
\,\,\rmd\sigma_{\!g} \,\,\, = \nonumber \\
& = & - \, \beta \,\, {\rm e}^s \!\!\!\!
\int\limits_{\{\ffi> s\}} \!\!
\frac{  |\na \ffi|_\g^{\beta-2}  
\Big(   \,  \big|\nana \ffi\big|_\g^2
 +  \, (\beta-2) \, \big| \na |\na \ffi |_\g   \big|_\g^2   \, \Big)       }{{\rm e}^{\ffi}}
\,\,\rmd\mu_g \,
\,\,\leq\,\, 0\,.
\end{eqnarray}
As for the case $\beta=2$, the monotonicity and rigidity results are obtained thanks to the nonnegativity of the rightmost hand side of the above identity, which is ensured by the standard Kato inequality when $\beta \geq 1$ and by the refined Kato inequality for harmonic functions when 
$(n-2)/(n-1)\leq\beta<1$. The monotonicity formulas claimed in the statement of Theorem~\ref{thm:main} are finally obtained via the identities
\begin{align*}
F_\beta(\tau) 
  \,\,=\,\, 
\Phi_\beta\left( \log \tau \right) \,  
\qquad \hbox{and} \qquad
 \tau\,\, F_\beta'(\tau) 
\,\,=\,\,  
 {\Phi_\beta'\left(\log \tau \right)} \, . 
\end{align*}


\subsection{Concluding remarks and further directions. } 
\label{sub:further}


We conclude this Introduction with some comments and  suggestions about the possible implications of our monotonicity formulas. An interesting feature is that they seem to indicate that in some specific contexts the level set flow of a suitably chosen harmonic function can be employed as a valid substitute of the Inverse Mean Curvature Flow to obtain sharp geometric inequalities. The advantages in considering the first flow instead of the latter one are quite evident. In fact, all the issues concerning the long time existence of the Inverse Mean Curvature Flow and its prolongation beyond the singular times, 
are somehow instantaneously ruled out,
due to the existence theory for 
harmonic functions in exterior domains. Also concerning the short time existence the first approach reveals a better flexibility, since the mean convexity - which is necessary to start running the Inverse Mean Curvature Flow -  is definitely not needed. Finally, the nowadays well understood structure of the nonregular level sets of harmonic functions can be fruitfully employed to extend the validity of the monotonicity formulas beyond the singular times.

Further evidences of this phenomenon have already been exploited in~\cite{Ago_Maz_2}, where a new proof of the Riemannian Penrose Inequality is obtained in every dimension for asymptotically flat static metrics, as well as in the forthcoming~\cite{Ago_Fog_Maz}, where the same monotonicity-rigidity theory is carried out in the context of complete metrics with nonnegative Ricci curvature and maximal volume growth. More in general, it would be interesting in our opinion to  investigate how far the technique based on the level set flow of harmonic functions could be employed to re-discover some of the beautiful achievements of the Inverse Mean Curvature Flow theory, possibly simplifying the analysis. 

Another beautiful challenge that is somehow suggested by the \emph{conformal splitting technique} described above concerns the possibility of obtaining enhanced versions of inequality~\eqref{eq:will2}, under topological constraints for the surface. This would eventually yield a different approach towards the study of the well known Willmore Conjecture, recently solved in the affirmative by the outstanding work of Marques and Neves~\cite{Mar_Nev}. 
In this direction we just mention that the exterior boundary value problem that is at the basis of our construction should be chosen differently from~\eqref{eq:pb}. However, looking at the model situation, it is not difficult to figure out some natural candidates. We plan to explore this path in future work.  


\subsection{Plan of the paper.} 
\label{sub:plan}

{
In Section~\ref{sec:conf_reform}, we reformulate problem~\eqref{eq:pb} according to the {\em cylindrical ansatz} described in Subsection~\ref{sub:strategy}. This leads to the new problem~\eqref{eq:pb_reform}, 
as well as to the conformal version of Theorem~\ref{thm:main}, namely to Theorem~\ref{thm:main_conf} below.
It is then shown how Theorem~\ref{thm:main} can be deduced after Theorem~\ref{thm:main_conf}. The remaining part of the section is devoted to the proof of some preliminary results for the analysis of system~\eqref{eq:pb_reform}, such as the gradient estimate of Proposition~\ref{thm:sharp_bound_fi} and the subsequent upper bound for the quantities $\Phi_\beta$'s,
which is the content of 
Lemma~\ref{unif_bound}. Section~\ref{sec:integral} contains the core of our analysis, which consists in the proof of two integral identities. The First Integral identity is proven in Proposition~\ref{prop:byparts} and  subsequently used in Corollary~\ref{cor:contfi} to deduce the continuity of the functions $\Phi_\beta$'s, according to the statement of Theorem~\ref{thm:main_conf}-(i). The Second Integral Identity is proven in Lemma~\ref{lem:cyl_reg} and then used in Corollary~\ref{cor:diffi} to deduce the differentiability and the monotonicity of the $\Phi_\beta$'s,
which are stated in Theorem~\ref{thm:main_conf}-(ii). The proof of Theorem~\ref{thm:main_conf} is finally completed with the Rigidity Statement deduced in Corollary~\ref{cor:rep-rig}. Section~\ref{consequences} is devoted to the consequences of Theorem~\ref{thm:main}. These are divided into 'Consequences at the boundary' (Subsection~\ref{sub:con_bound}) and 'Global geometric consequences' (Subsection~\ref{sub:con_glob}). The first include some new geometric upper bounds for the Capacity in terms of the $L^p$-norm of the mean curvature of the boundary (see Corollary~\ref{cor:capupper}), whereas the latter 
include the geometric inequalities described in Theorem~\ref{cor:will-def} and Theorem~\ref{cor:mink-def}.} 

\bigskip

\begin{center}
{\bf Added note}
\end{center}

\smallskip

{

The present manuscript is a new version of a 
paper posted on ArXiv in June 2016. 
The main novelties and perspectives 
with respect to the old version are described 
in the following list.

\begin{itemize}

\item On top of the list are the new quantitative versions of the Willmore-type inequality~\eqref{eq:will}
and the weighted Minkowski 
inequality~\eqref{eq:minki} for general domains.
Let us stress the fact that these results
are obtained exploiting the full power of
Theorem~\ref{thm:main}, up to the threshold
value $\beta=(n-2)/(n-1)$. In the previous version of the manuscript the main geometric corollaries were deduced only using values of the parameter $\beta$ above $1$ ($p\geq 2$, according to the old notation).

\smallskip

\item The proof of the Monotonicity-Rigidity Theorem for $(n-2)/(n-1) \leq \beta <1$ contained in the old version of the manuscript was more involved and heavily relied on some fine properties of the critical set of the harmonic functions. {\em In primis} on the fact that the Hausdorff dimension of such a set is at most $(n-2)$. In this version, we present a new argument, which has main the advantage of being self-contained. As such, it does not require any {\em a priori} knowledge on the size of the set where the velocity of the level set flow blows-up. For these reasons this argument can be adapted to treat the case of the level set flow of $p$-harmonic functions. This is the content of the forthcoming~\cite{Ago_Fog_Maz_2}.

\smallskip

\item The monotonicity formulas obtained in Theorem~\ref{thm:main}
have a natural extension to harmonic functions defined
in exterior domains on manifolds with nonnegative
Ricci curvature.
This context is investigated in~\cite{Ago_Fog_Maz}, where, also,
a systematic comparison is drawn between our monotonicity
formulas and those obtained for Green's functions by Colding and Minicozzi 
in the beautiful papers~\cite{Colding_1,Colding_Minicozzi_2,Colding_Minicozzi}.
Here we just observe that in the present setting their formulas 
do not yield any new information, due to the rotational symmetry of the
Euclidean Green's function. 

\smallskip

\item The above cited papers of Colding and Minicozzi have also 
inspired the new notation $F_\beta$, with $\beta\geq(n-2)/(n-1)$, 
in place of the old $U_p$, with $p\geq2-1/(n-1)$
(already used in~\cite{Ago_Maz_3} and in~\cite{Ago_Maz_2}).
Another reason for this change of notation is that
it allows to rephrase 
the property stated in Theorem~\ref{thm:main}-(iii)
in terms of the convexity of the functions $F_\beta$'s. 

\smallskip

\item Concerning the analysis of problem~\eqref{eq:pb_reform}, 
the hypothesis of \emph{bounded geometry} 
contained in the
old version of the paper is now replaced by the
much less restrictive growth condition
~\eqref{hyp_migliore}. This is possible
since the latter assumption is sufficient to deduce 
the sharp gradient bound~\eqref{gra_est},
which implies the boundedness of
the functions $\Phi_\beta$'s (see Lemma \ref{unif_bound}) and
in turn the monotonicity, as outlined at the end of 
Section~\ref{sec:conf_reform}.

\end{itemize}


\section{A conformally equivalent setting}
\label{sec:conf_reform}


\subsection{A conformal change of metric.}
\label{sub:conf}

Proceeding as in ~\cite[Section 2]{Ago_Maz_1},
we perform a conformal change of the Euclidean metric
to obtain an equivalent formulation of problem~(\ref{eq:pb}). 
Consider a solution $u$ to problem~(\ref{eq:pb}) 
and note that $0<u<1$, 
by the maximum principle.
To set up the notation, we let $M$ be $\R^n\setminus\Om$, 
denote by $g_{\R^n}$ the flat Euclidean metric of $\R^n$,
and consider the conformally equivalent metric given by
\begin{equation}
\label{eq:tilde_g}
g \, = \, u^{\frac2{n-2}}g_{\R^n}.
\end{equation}
To reformulate our problem it is also convenient to set 
\begin{equation}
\label{def:ffi}
\ffi \, = \, -\log u, 
\end{equation}
so that the metric $g$ can be equivalently written as
$g={ e}^{-\frac{2\ffi}{n-2}}g_{\R^n}$.
In what follows we denote by 
$\langle \cdot | \cdot\cdot \rangle$ and 
$\langle \cdot | \cdot\cdot \rangle_g$ and by 
$\DDD$ and $\na$ the scalar products and the 
covariant derivatives of the metrics $g_{\R^n}$ and $g$, respectively. 
The symbols $\D\D$, $\na\na$ and $\Delta$, $\Delta_g$ 
stand for the corresponding Hessian and Laplacian operators. 
{\em In the sequel, for the sake of simplicity, we will refer to a function 
in the kernel of $\Delta_g$ as to a $g$-harmonic function}.
The same computation as in~\cite[Section 2]{Ago_Maz_1} 
show that problem~\eqref{eq:pb} is equivalent to
\begin{equation}
\label{eq:pb_reform}
\left\{
\begin{array}{rcll}
\displaystyle
\phantom{\frac12}\Delta_g \ffi \!\!\!\!&=&\!\!\!\!0 & {\rm in }\quad M,\\
\displaystyle
\Ric_g - \na\na\ffi +\frac{d\ffi\otimes d\ffi}{n-2}
\!\!\!\!&=&\displaystyle\!\!\!\!\frac{|\na \ffi|^2_g}{n-2}\,g & {\rm in }\quad M,\\
\displaystyle
 \phantom{\frac12} \ffi\!\!\!\!&=&\!\!\!\!0  &{\rm on }\ \ \pa M,\\
\displaystyle
\phantom{\frac12}\ffi(x)\!\!\!\!&\to&\!\!\!\!+\infty & \mbox{as }\ x\to\infty.
\end{array}
\right.
\end{equation}
Notice that the second equation corresponds to the 
understatement $\Ric_{g_{\R^n}} = 0$, 
which is implicit in the fact that the background metric 
in problem~\eqref{eq:pb} is the flat one.


\subsection{The extrinsic curvature of the level sets.}


In the forthcoming analysis it will be important to study the geometry of the level sets of $\ffi$, which coincide with the level sets of $u$ by definition. To this end, we denote by ${\rm Crit}(\ffi)$ the set of the critical points of $\ffi$ and we fix on $M \setminus {\rm Crit}(\ffi)$ the $g_{\R^n}
$-unit vector field $\nu=-\DDD u/|\DDD u|=\DDD \ffi/|\DDD \ffi|$ and the $g$-unit vector field $\nu_g=-\na u/|\na u|_g=\na \ffi/|\na \ffi|_g$. Consequently, the second fundamental forms of the regular level sets of $u$ or $\ffi$ with respect to the flat ambient metric and the conformally-related ambient metric $g$, are given by
\begin{align*}
\ho(X,Y)& \,= \, -\frac{\D\D u(X,Y)}{|\DDD u|} \, = \,\frac{\D\D\ffi(X,Y)}{|\DDD \ffi|},\\
\ho_g(X,Y)& \, = \, -\frac{\na\na u(X,Y)}{|\na u|_g} \, = \, \frac{\na\na\ffi(X,Y)}{|\na\ffi|_g},
\end{align*}  
respectively, where $X$ and $Y$ are vector fields tangent to the level sets. Taking the traces of the above expressions with respect to the induced metrics and using the fact that $u$ is harmonic and $\ffi$ is $\g$-harmonic, we obtain the following expressions for the mean curvatures in the two ambients
\begin{equation}
\label{eq:formula_curvature}
\Ho\, = \, \frac{\D\D u(\D u,\D u)}{|\D u|^3}\,,
\qquad\quad
\Hg \, = \, - \, \frac{\nana\ffi(\na\ffi,\na\ffi)}{|\na\ffi|_{\g}^3}\,.
\end{equation}
By a direct computation one can show that the second fundamental forms and the mean curvatures are related by the following formul\ae
\begin{align}
\label{eq:formula_h_h_g}
\ho_g(X,Y) 
&\,=\,
u^{\frac1{n-2}}
\bigg[ \, \ho (X,Y) \, - \,   \Big( \frac{1}{n-2}  \Big) \, 
\frac{| \D u | }{u}
\,  \langle X |Y \rangle  \bigg] \, , \\
\label{eq:formula_H_H_g}
{\HHH_g}
&\,=\,
u^{-\frac1{n-2}}
\bigg[\, {\Ho} - \Big( \frac{n-1}{n-2}    \Big)  
\, 
\frac{ |\D u | } u 
\, \bigg] \, ,
\end{align}
where, as before, 
$X$ and $Y$ are vector fields tangent to the level sets. For the sake of completeness, we also report the reverse formul\ae\
\begin{align*}
\ho(X,Y)
&\,=\,
e^{\frac{\ffi}{n-2}}
\bigg[ \, \ho_g (X,Y)
+ \Big(\frac{1}{n-2}\Big) \,|\na \ffi|_\g \,  \langle X |Y \rangle_g \, \bigg] \, , \\
{\Ho}
&\,=\, e^{-\frac{\ffi}{n-2}}
\bigg[\, {\Hg}   + \Big(\frac{n-1}{n-2}\Big)\,  |\na \ffi|_\g  \,\bigg] \, .
\end{align*}

Concerning the nonregular level sets of $\ffi$, 
we just observe that, 
by the same arguments as in Remark \ref{rem:uno_u},
the second fundamental form and mean curvature also make sense 
$\mathscr{H}^{n-1}$-almost everywhere on
a singular level set$\{ \ffi = s_0 \}$, 
namely on the relatively open set 
$\{ \ffi =  s_0\} \setminus {\rm Crit}(\ffi)$.


\subsection{A conformally equivalent version of the Monotonicity-Rigidity Theorem.}
\label{sub:computations}

In order to take advantage of our {\em cylindrical ansatz}, it is convenient to reformulate the statement of Theorem~\ref{thm:main} in the new conformally related setting. To do that, we introduce, for every $\beta \geq 0$,
the function
$\Phi_\beta: [0, +\infty) \longrightarrow \R$, setting
\begin{equation}
\label{eq:fip}
\Phi_\beta(s) 
\,\,:=\!\!\!
\int\limits_{\{\ffi = s\}}\!\!\!
|\na \ffi|_g^{\beta+1} \,\,\rmd \sigma_{\!g} \, ,
\end{equation}
so that if $(u, g_{\R^n})$ and $(\ffi, g)$ are related by formula\ae~\eqref{eq:tilde_g} and~\eqref{def:ffi}, then the following relationships are also in force
\begin{align*}
F_\beta(\tau) 
  \,\,=\,\, 
\Phi_\beta\left( \log \tau \right) \,  
\qquad \hbox{and} \qquad
 \tau\,\, F_\beta'(\tau) 
\,\,=\,\,  
 {\Phi_\beta'\left(\log \tau \right)} \, . 
\end{align*}
%
%

\noindent Having this in mind, we can now re-state the first two points in Theorem~\ref{thm:main} in the following way.

\begin{theorem}[Monotonicity-Rigidity Theorem -- Conformal Version]
\label{thm:main_conf}
Let $(M,g,\ffi)$ be a solution to problem~\eqref{eq:pb_reform}
such that $\pa M$ is a regular level set of $\ffi$,
and assume that  the following growth condition
\begin{equation}
\label{hyp_migliore}
|\na\ffi|_g^2(x) \,\, = \,\, {o}\, ({\rm e}^\ffi) \, 
\qquad\mbox{as}\quad x\to\infty 
\end{equation}
is satisfied.
Let $\Phi_\beta : [0,+\infty) \longrightarrow \R$ 
be the function defined  in~\eqref{eq:fip}. Then, the following properties hold true.
\begin{itemize}
\item[(i)] \underline{Continuity.} For every $\beta\geq0$, 
the function $\Phi_\beta$ is continuous and admits for every $s \geq 0$ the integral representation 
\begin{equation*}
\Phi_\beta(s)\,\,
= \,\,\,{\rm e}^s \!\!\!\!\int\limits_{\{\ffi>s\}}
\!\!
\frac{  |\na \ffi|_g^{\beta-2}  
\Big(\,|\na\ffi|_\g^4 -  \beta\,  \nana \ffi (\na\ffi, \!\na\ffi)   \Big) }{{\rm e}^\ffi}
\,\,\rmd\mu_\g\,.
\end{equation*}
%

\smallskip

\item[(ii)] \underline{Differentiability, Monotonicity \& Rigidity.} For every $\beta\geq(n-2)/(n-1)$, 
the function $\Phi_\beta$ is continuously differentiable and the derivative $\Phi_\beta'$ admits for every $s \geq 0$ the integral representation
\begin{equation}
\label{eq:mono_fip}
\Phi_\beta'(s) \,\, 
= \,\,-\beta\,\,{\rm e^s}\!\!\!
\int\limits_{\{\ffi> s\}} \!\!
\frac{  |\na \ffi|_\g^{\beta-2}  
\Big(   \,  \big|\nana \ffi\big|_\g^2
 +  \, (\beta-2) \, \big| \na |\na \ffi |_\g   \big|_\g^2   \, \Big)       }{{\rm e}^{\ffi}}
\,\,\rmd\mu_g \,
\,\leq\,0\,. 
\end{equation}
Moreover, if there exists $s_0\geq0$ such that $\Phi_\beta'(s_0) = 0$, for some $\beta_0\geq(n-2)/(n-1)$,
then the manifold $(\{ \ffi \geq s_0\} , g)$ is
isometric to 
$\big(\,[s_0,+\infty)\times\{\ffi=s_0\},
d\varrho \otimes d\varrho + \g_{|\{ \ffi = s_0 \}}\big)$,
where $\varrho$ is the distance to $\{ \ffi = s_0\}$,
and $\ffi$ is an affine function of $\varrho$.
\end{itemize}
\end{theorem}
The above statement will be proven in Section~\ref{sec:integral}.
More precisely, Theorem~\ref{thm:main_conf}-(i) will be deduced in Subsection~\ref{sub:first} (Corollaries~\ref{cor:repfi} and~\ref{cor:contfi}) as a consequence of the First Integral Identity~\eqref{eq:id_byparts_bis}, whereas Theorem~\ref{thm:main_conf}-(ii) will be proven in Subsection~\ref{sub:second} (Corollaries~\ref{cor:diffi} and~\ref{cor:rep-rig}) with the help of the Second Integral Identity~\eqref{eq:id_byparts_reg}.

In complete analogy with Remark~\ref{rem:deretta}, it is worth stating the following 
\begin{remark}
\label{rem:deretta_g}
We observe that if $s$ is a regular value of $\ffi$, then the function $\Phi_\beta$ is differentiable at $s$ for every $\beta \geq 0$, and a direct computation gives 
\begin{equation}
\label{derivata_di_Phi}
\Phi_\beta'(s) 
\,\,=\,\, 
- \, \beta\!\!\!\! \int\limits_{\{\ffi = s \}}\!\!\!\!  
{\, |\na\ffi|_\g^\beta\, \Hg} \,\rmd\sigma_\g \,,
\end{equation}
where $\HHH_g$ is the mean curvature of the level set $\{\ffi=s\}$
computed with respect to the unit normal vector field 
$\nu_g=\na \ffi/|\na \ffi|_g$. 
\end{remark}
The above observation will be employed in the discussion of the borderline case $\beta = (n-2)/(n-1)$ in the proof of the Rigidity Statement in Corollary~\ref{cor:rep-rig} below.

We conclude this subsection showing how our main Theorem~\ref{thm:main} can be deduced from Theorem~\ref{thm:main_conf}.

\begin{proof}[proof of Theorem~\ref{thm:main}, after Theorem~\ref{thm:main_conf}]
First of all, we notice that the classical asymptotic expansions~\eqref{eq:u_exp} imply through formula~\eqref{eq:P_function} that 
$|\na \ffi|_g^2 = \mathcal{O} (1)$, as $x \to \infty$ and thus the growth condition~\eqref{hyp_migliore} is fulfilled.
Now, we observe that the continuity statement in Theorem~\ref{thm:main}-(i) is a straightforward consequence of the analogous statement in Theorem~\ref{thm:main_conf}-(i), once the relationship $F_\beta(\tau) 
\,=\,
\Phi_\beta\left( \log \tau \right)$ is taken into account.

The differentiability and the monotonicity of the $F_\beta$'s for $\beta \geq (n-2)/(n-1)$ in Theorem~\ref{thm:main}-(ii), follow from the analogous statements for the $\Phi_\beta$'s in Theorem~\ref{thm:main_conf}-(ii), using the relationship $\tau\,\, F_\beta' (\tau) \,=\, \Phi_\beta' \left(\log \tau \right)$. 
In particular, the Monotonicity Formula~\eqref{eq:der_fb} is 
the reformulation of~\eqref{eq:mono_fip} in terms of $u$
and of the Euclidean metric, via the identities
\begin{align*}
|\na\na\ffi|^2
&\,=\,
u^{-\frac4{n-2}}
\left\{
\left|\frac{\D\D u}u\right|^2
+\frac{n(n-1)}{(n-2)^2}\left|\frac{\D u}u\right|^4
-\Big(\frac{2n}{n-2}\Big)\left|\frac{\D u}u\right|^3\!\!{\rm H}
\right\}\\
\big|\na|\na\ffi|\big|^2
&\,=\,
u^{-\frac4{n-2}}
\left\{
\left|\frac{\D|\D u|}u\right|^2
+\Big(\frac{n-1}{n-2}\Big)^2\left|\frac{\D u}u\right|^4
-2\,\big(\frac{n-1}{n-2}\Big)\left|\frac{\D u}u\right|^3\!\!{\rm H}
\right\},
\end{align*}
which yield, by adding and subtracting the term
$u^{-4/(n-2)}[n/(n-1)]\big|\D|\D u|/u\big|^2$,
\begin{align*}
|\na\na\ffi|^2
+
(\beta-2)
\big|\na|\na\ffi|\big|^2
\,=\,
u^{-\frac{2n}{n-2}}
&
\Bigg\{
\left|\D\D u\right|^2-\Big(\frac n{n-1}\Big)\big|\D|\D u|\big|^2\\
&+
\Big(\beta-\frac{n-2}{n-1}\Big)\big|\D^T|\D u|\big|^2\\
&+
\Big(\beta-\frac{n-2}{n-1}\Big)
|\D u|^2
\bigg[
{\rm H}-\Big(\beta-\frac{n-1}{n-2}\Big)
\left|\frac{\D u}u\right|\,
\bigg]^2
\Bigg\}.
\end{align*}
The Rigidity Statement follows from Remark~\ref{rem:analytic}. Indeed, the solution $u$ to problem~\eqref{eq:pb} is analytic and hence the metric $g$ defined through~\eqref{eq:tilde_g} is 
analytic as well.

The convexity of the functions $F_\beta$'s in the statement of Theorem~\ref{thm:main}-(iii) follows from the fact that the assignment $\tau \mapsto F_\beta'(\tau)$ is nondecreasing. In fact, for any given $1 \leq \tau_1 \leq \tau_2 < + \infty$, one has that
\begin{align*}
F_\beta'(\tau_2) \, - \,  F_\beta'(\tau_1) \,\,& = \,\, \frac{\Phi_\beta'(\log \tau_2)}{\tau_2} \, - \, \frac{\Phi_\beta'(\log \tau_1)}{\tau_1} \,\, = \\
& = \,\,\,\beta \!\!\!\!\!\! \!\!\!\!\!\!\!
\int\limits_{\{\log \tau_1<\ffi<\log\tau_2\}} \!\!\!\!\!\!\!\!\!\!\!\!
\frac{  |\na \ffi|_g^{\beta-2}  
\Big(   \,  \big|\nana \ffi\big|_g^2
 +  \, (\beta-2) \, \big| \na |\na \ffi |_g\big|_g^2   
 \, \Big) }
 {{\rm e}^\ffi}
\,\, \, \rmd\mu_\g \,,
\end{align*}
where the last equality corresponds to identity~\eqref{eq:premono} in Corollary~\ref{cor:diffi}, but can also be deduced at once from~\eqref{eq:mono_fip} in Theorem~\ref{thm:main_conf}-(ii). Finally, as already observed in Remark~\ref{rem:seconder}, the representation formula~\eqref{eq:der2_fb} for the second derivatives of the functions $F_\beta$'s is a straightforward consequence of the Coarea Formula.
\end{proof}


\subsection{Some pointwise identities and a sharp gradient estimate.}
\label{sub:computations}

This subsection is devoted to some preliminary computations and results, that will be readily applied in the forthcoming Section~\ref{sec:integral} to obtain a couple of integral identities and then the complete proof of Theorem~\ref{thm:main_conf}.

 We start by noticing that for a solution $(M, g,\ffi)$ of problem~\eqref{eq:pb_reform}, the Bochner formula reduces to the identity
\begin{equation}
\label{eq:Bochner_pot}
\Delta_\g|\na \ffi|_\g^2 \, - \, \big\langle\na|\na \ffi|^2_\g \, \big| \,\na \ffi \big\rangle_{\!\!\g} \,
= \,
2 \, |\na\na\ffi|^2_\g \, .
\end{equation} 
Now, observe that, wherever the following expressions are well defined, it holds
\begin{align*}
\na |\na \ffi|_\g^\beta 
& \,= \, 
\Big(\frac\beta2\Big)  
|\na \ffi|_\g^{\beta-2} \, \na |\na \ffi|_\g^2 \, ,\\
\Deg |\na \ffi|_\g^{\beta} 
&\, = \, 
\Big(\frac\beta2\Big)  
|\na \ffi|_\g^{\beta-2} \, {\Deg |\na \ffi|_g^2} 
+ \beta(\beta-2) \, |\na \ffi|_\g^{\beta-2} 
\, \big| \na |\na \ffi |_\g   \big|_\g^2 \,.
\end{align*}
Combining these two facts together with identity~\eqref{eq:Bochner_pot}, we arrive at
\begin{equation}
\label{eq:Bochner_pot_beta}
\Delta_\g|\na \ffi|_\g^\beta\, - \, 
\big\langle\na|\na \ffi|^\beta_\g \, \big| 
\,\na\ffi\big\rangle_{\!\!\g} 
\,\,=\, \, 
\beta \, |\na \ffi|_\g^{\beta-2}  
\Big[   \,  \big|\nana \ffi\big|_\g^2
 +  \, (\beta-2) \, 
 \big| \na |\na \ffi |_\g   \big|_\g^2   \, \Big]  \,.
\end{equation}

As we are going to see in the next section, the above relationship is at the core of our fundamental integral identities~\eqref{eq:id_byparts_bis} and~\eqref{eq:id_byparts_reg}, and thus also of the Monotonicity-Rigidity Theorem~\ref{thm:main_conf}, that from these identities is deduced. 

As a first effective application of the above computations, we are going to prove the following sharp gradient estimate 
(in the spirit of~\cite{Colding_1}), 
that readily translates into Theorem~\ref{thm:sharp_bound_u}, when $(u, g_{\R^n})$ and $(\ffi, g)$ are related by formul\ae~\eqref{eq:tilde_g} and~\eqref{def:ffi}.


\begin{proposition}
\label{thm:sharp_bound_fi}
Let $(M,g,\ffi)$ be a solution to problem~\eqref{eq:pb_reform} and assume that  the growth condition
\begin{equation*}
|\na\ffi|_g^2 \,\, = \,\, o\, ({\rm e}^\ffi) \, ,
\qquad\mbox{as}\quad x\to\infty, 
\end{equation*}
is satisfied. Then, for every $x \in M$ it holds
\begin{equation}
\label{gra_est}
\,|\na\ffi|_g^2(x)
\,\,\leq\,\,
\max_{\pa M}|\na\ffi|_g^2 \,. \qquad\qquad \qquad
\end{equation}
Moreover, the equality is achieved at 
some interior point of $M$
if and only if
$(M,\g,\ffi)$ is isometric to one half round cylinder.
\end{proposition}
\begin{proof}
Setting $K = \max_{\pa M} |\na \ffi|_g$ and 
\begin{equation*}
w \,\,=\,\,
\frac{K^2-|\na\ffi|_\g^2}{{\rm e}^\ffi} \, ,
\end{equation*}
it is immediate to check, with the help of~\eqref{eq:Bochner_pot}, that 
\begin{equation}
\label{eq_w_beta_2}
\Delta_\g w
\, + \,
\langle\na w|\na\ffi\rangle_\g
\,\,=\,\,
-\, 
2\,{\rm e}^{-\ffi} \, |\na\na\ffi|^2_\g
\,\,\leq\,\,
0 \, .
\end{equation}
Now observe that the growth condition~\eqref{hyp_migliore} ensures that $w(x)\to0$ as $x\to\infty$. On the other hand, it follows from the definition of $K$ that $w\geq 0$ on $\pa M$. The desired gradient bound~\eqref{gra_est} is now an easy consequence of the Maximum Principle. Finally, the rigidity statement follows from the Strong Maximum Principle, which in turn implies
the vanishing of $\na\na\ffi$ through equation~\eqref{eq_w_beta_2}.
\end{proof}

In the next lemma we are going to show that the gradient estimate just obtained forces the functions $\Phi_\beta$'s defined in~\eqref{eq:fip} to be bounded. We then conclude this subsection with an outline of the proof, under favourable assumptions, of the monotonicity claimed in Theorem~\ref{thm:main_conf}. This argument is inspired by the corresponding monotonicity results for the pseudoarea and pseudovolume functionals in~\cite{Colding_1}.

\begin{lemma}
\label{unif_bound}
Let $(M,g,\ffi)$ be a solution to problem~\eqref{eq:pb_reform} and assume that  the growth condition $|\na\ffi|_g^2 \, = \, o\, ({\rm e}^\ffi)$ is satisfied, as $x \to \infty$.
Then, for every $\beta\geq 0$, there exists a constant $C_\beta>0$ such that
\begin{equation*}
\Phi_\beta(s)\,\,\leq\, \,C_\beta \, , 
\end{equation*}
for every $s \geq 0$.
\end{lemma}

\begin{proof}
We have that
\begin{equation*}
\Phi_\beta(s)
\,\,=\!\!\!
\int\limits_{\{\ffi=s\}}
\!\!\! |\na\ffi|_\g^\beta
\, |\na\ffi|_\g \, \rmd\sigma_\g
\,\,\leq\,\,
\Big(\!\max_{\pa M}|\na\ffi|_\g^\beta \Big)
\!\!\!\int\limits_{\{\ffi=s\}}
\!\!\!|\na\ffi|_\g \, \rmd\sigma_\g
\,\,=\,\,\Big(\!\max_{\pa M}|\na\ffi|_\g^\beta \Big)
\!\int\limits_{\pa M}
\!|\na\ffi|_\g \, \rmd\sigma_\g \, ,
\end{equation*}
where the last equality follows from the Divergence Theorem and from the fact $\ffi$ is $g$-harmonic.
\end{proof}


\begin{proof}[Outline of the proof of the Mononotonicity in Theorem~\ref{thm:main_conf}-(ii)]
Suppose to have a solution $(M,g,\ffi)$ to problem~\eqref{eq:pb_reform} satisfying the requirement~\eqref{hyp_migliore}, so that the above lemma is in force. Assume in addition that  
the following integral identity holds true for every $\beta\geq(n-2)/(n-1)$ and every $0\leq S_0<S$
\begin{equation*}
\int\limits_{\{\ffi=S_0\}} \!\!\!\!
\frac{   |\na \ffi|_\g^\beta\,\HHH_\g }{ {\rm e}^{S_0} }
\,\,\rmd\sigma_\g  
\,  - \!\!\! \int\limits_{\{\ffi=S\}}\!\!\!\!
\frac{   |\na \ffi|_\g^\beta\,\HHH_\g }{ {\rm e}^S }
\,\,\rmd\sigma_\g
\,\, = \!\! 
\int\limits_{\{S_0<\ffi<S\}} \!\!\!\!\!
\frac{  |\na \ffi|_\g^{\beta-2}  
\Big(   \,  \big|\nana \ffi\big|_\g^2
 +  \, (\beta-2) \, \big| \na |\na \ffi |_\g\big|_\g^2   
 \, \Big) }
 {{\rm e}^\ffi}
\,\,\rmd\mu_\g  \,,
\end{equation*}
where $\HHH_g$ is the mean curvature of the level set computed with respect to the unit normal vector field $\na \ffi/|\na\ffi|_g$.
In the forthcoming Lemma~\ref{lem:cyl_reg} this identity will be completely justified whenever $S_0$ and $S$ are regular values,
and then it will be used in Corollary~\ref{cor:diffi} to prove that the function 
$\Phi_\beta$ 
is differentiable and satisfies
\begin{equation}
\label{basic_equiv}
\Phi_\beta' (S) \,  {\rm e}^{-S} \,\, 
-\,\,
\Phi_\beta' (S_0) \, {\rm e}^{-S_0}
\,= \,\,\,\beta \!\!\!\!\!\! \!
\int\limits_{\{S_0<\ffi<S\}} \!\!\!\!\!\!
\frac{  |\na \ffi|_\g^{\beta-2}  
\Big(   \,  \big|\nana \ffi\big|_\g^2
 +  \, (\beta-2) \, \big| \na |\na \ffi |_\g\big|_\g^2   
 \, \Big) }
 {{\rm e}^\ffi}
\,\, \, \rmd\mu_\g \,.
\end{equation}
In particular, since the right-hand side of the 
above identity is nonnegative in view
of the refined Kato inequality for 
harmonic functions, we have that
\begin{equation*}
\Phi_\beta'(S)  \,\, 
 \geq \,\,
\Phi_\beta'(S_0) \,\, {\rm e}^{S-S_0} \, .
\end{equation*}
Now, recall that our claim in Theorem~\ref{thm:main_conf} is that the derivative of $\Phi_\beta$ is everywhere nonpositive. Suppose then by contradiction that
$\Phi_\beta'  (S_0)>0$ for some $S_0\geq0$.
Integrating the above differential inequality gives
\begin{equation*}
\Phi_\beta(S)
\,\,\geq \,\,
\Phi_\beta(S_0) \, + \,  \Phi_\beta'(S_0) \, 
 \, \frac{{\rm e}^S\!- \,{\rm e}^{S_0}}{{\rm e}^{S_0}} \, ,
\end{equation*}
for every $0\leq S_0 <S$.
In turn, we would get $\Phi_\beta(S)\to+\infty$
as $S\to+\infty$, contradicting Lemma \ref{unif_bound}.
This provides the desired monotonicity.
\end{proof}

The crucial role played by integral identities in this very simple argument, motivates the analysis of the following sections.

{


\section{Integral identities and proof of Theorem~\ref{thm:main_conf}}
\label{sec:integral}

In this section we derive a couple of integral identities that will then be used to prove a conformally equivalent version of the Monotonicity-Rigidity Theorem~\ref{thm:main_conf}. They are essentially deduced using identity~\eqref{eq:Bochner_pot_beta}
\begin{equation*}
\Delta_\g|\na \ffi|_\g^\beta\, - \, 
\big\langle\na|\na \ffi|^\beta_\g \, \big| 
\,\na\ffi\big\rangle_{\!\!\g} 
\,\,=\, \, 
\beta \, |\na \ffi|_\g^{\beta-2}  
\Big[   \,  \big|\nana \ffi\big|_\g^2
 +  \, (\beta-2) \, 
 \big| \na |\na \ffi |_\g   \big|_\g^2   \, \Big]  \,
\end{equation*}
and applying the Divergence Theorem to the vector fields
\begin{equation*}
X  \,= \,  \frac{|\na\ffi|_g^\beta\na \ffi}{{\rm e}^{\ffi}} \qquad \hbox{and} \qquad Y  
\,=\,  
\frac{\na|\na\ffi|_g^\beta}{{\rm e}^{\ffi}}\, .
\end{equation*}
on sets of the form 
\begin{equation} 
\label{set_E}
E_s^S
\,:=\, 
\{ s < \ffi < S\}\,.
\end{equation}
Care will be needed to justify
the integration by parts when critical points
of $\ffi$ are present in $\overline E_s^S$.
To this aim, it will be convenient to
consider a suitable family of neighbourhoods 
of the set of critical points
${\rm Crit} (\ffi)$. Whenever ${\rm Crit} (\ffi)$ is nonempty, we set for $\ep>0$
\begin{equation}
\label{eq:pigiama}
U_\ep
\,:=\,
\big\{|\na\ffi|_\g^2<\ep\big\} \, \supseteq {\rm Crit} (\ffi)  \,.
\end{equation}
We recall that since $|\na\ffi|_\g^2$ is smooth,
Sard's Theorem implies that 
for almost every $\ep>0$ the boundary 
$\pa U_\ep$ is a regular level set of 
$|\na\ffi|_\g^2$.
In particular,
$\pa U_\ep$ is a smooth $(n-1)$-dimensional
hypersurface with 
\begin{equation}
\label{eq:nor_uep}
\frac{\na|\na\ffi|_\g^2}
{\big|\na|\na\ffi|_\g^2\big|_\g}
\end{equation}
as a unit normal vector field.
Since in the present situation $\ffi$ is analytic and thus $|\na\ffi|_\g^2$ is analytic as well, 
the result of \cite{Sou_Sou} 
guarantees that all the properties of $\pa U_\ep$ described above hold true
except for a discrete -- and thus locally finite -- set of values of the parameter $\ep$.


\subsection{First integral identity and proof of Theorem~\ref{thm:main_conf}-(i).}
\label{sub:first}

Having fixed the set-up, we are now ready to prove our integral identities. The first integral identity will be directly employed to prove that the assignment $s\mapsto\Phi_\beta(s)$ is continuous for every $\beta \geq 0$. 

\begin{proposition}[First Integral Identity]
\label{prop:byparts}
Let $(M,g,\ffi)$ be a solution to 
problem~\eqref{eq:pb_reform} and assume that  the growth condition $|\na\ffi|_g^2 \, = \, o\, ({\rm e}^\ffi)$ is satisfied, as $x \to \infty$. Then, for every 
$\beta\geq0$ and every $0 \leq s < S$, 
we have
\begin{equation}
\label{eq:id_byparts_bis}
  \int\limits_{\{\ffi=S\}}\!\!\!\!
\frac{   |\na \ffi|_\g^{\beta+1}}{ {\rm e}^{S} }
\,\,\rmd\sigma_\g \,  - \!\!\!
 \int\limits_{\{\ffi=s\}} \!\!\!\!
\frac{   |\na \ffi|_\g^{\beta+1}}{ {\rm e}^{s} }
\,\,\rmd\sigma_\g  
\,\,= \!\!\!\!\! \!\!
\int\limits_{\{s<\ffi<S\}} \!\!\!\!\!\!\!
\frac{  |\na \ffi|_\g^{\beta-2}  
\Big(\, \beta\,  \nana \ffi (\na\ffi, \!\na\ffi)  \, - \, |\na\ffi|_g^{4} \,\Big) }
 {{\rm e}^\ffi}
\,\,\rmd\mu_\g  \,,
\end{equation} 
\end{proposition}

\begin{proof}
To simplify the notation, we drop the subscript $\g$ throughout 
the proof. 
We consider the vector field 
\begin{equation*}
X  \,= \,  \frac{|\na\ffi|^\beta\na \ffi}{{\rm e}^{\ffi}},
\end{equation*}
and compute, wherever $|\na\ffi| > 0$,
\begin{equation*}
{\rm div}X\, = \, \frac{  |\na \ffi|^{\beta-2}  
\Big(\beta\,  \nana \ffi (\na\ffi, \!\na\ffi) \, 
- \, |\na\ffi|^{4}  \Big) }{{\rm e}^{\ffi}} \, ,
\end{equation*}
using the fact that $\ffi$ is harmonic. 
Notice that the harmonicity of $\ffi$ also implies that ${\rm div}X$ is locally bounded on $M$, for every $\beta \geq 0$. Now, for $0 \leq s < S$, we consider the set $E_s^S$ defined in~\eqref{set_E}
and observe that if the closure of $E_s^S$ does not contain any critical points of $\ffi$, then the thesis follows directly from the Divergence Theorem. On the other hand, we recall from~\cite{Sou_Sou} that there exists at most a finite number of critical values for $\ffi$ in between $s$ and $S$. To simplify the argument, let us now assume without loss of generality that there is only one critical value $\bar s$ of $\ffi$ in $(s,S)$, otherwise it is sufficient to repeat the same argument a finite number of times. Notice that the level sets $\{\ffi = s\}$ and $\{ \ffi = S\}$ might be critical as well.
For small enough values of the parameter $\ep$, let $\{U_\ep\}_\ep$ 
be the (nondecreasing) family of 
\emph{tubular neighbourhoods}
of $ {\rm Crit} (\ffi)$ defined in~\eqref{eq:pigiama}.
The Divergence Theorem applied
to the vector field $X$ on $E_s^S \setminus U_\ep$ gives 
\begin{equation*}
\int\limits_{E_s^S\setminus U_\ep}
\!\!\!
{\rm div}X\,\,\rmd\mu
\,\,\,\,= \!\!\!\!
\int\limits_{\pa(E_s^S\setminus U_\ep)}
\!\!\!\!\!\!\!
\big\langle X
\,\big|\,{\rm n}\big\rangle\,\,\rmd\sigma \, ,
\end{equation*}
where ${\rm n}$ is the outer unit normal to the boundary of $E_s^S\setminus U_{\ep}$. 
In particular, ${\rm n}$ is well defined almost everywhere on the boundary and it is given by formula~\eqref{eq:nor_uep} on $\pa U_\ep$, whereas it coincides with $ \pm \na \ffi / |\na \ffi|$ on $\pa E_s^S \setminus \pa U_\ep$.
To prove~\eqref{eq:id_byparts_bis}, we first observe that, since ${\rm div}X$ is locally bounded for every $\beta \geq 0$, the Dominated Convergence Theorem gives 
\begin{align}
\lim_{\ep\downarrow0}
\int\limits_{E_s^S\setminus U_\ep}
\!\!\!
{\rm div}X\,\,\rmd\mu
\,\,\,= \,\int\limits_{E_s^S}\,
&\frac{  |\na \ffi|^{\beta-2}  
\Big(\beta\,  \nana \ffi (\na\ffi, \!\na\ffi) \, 
- \, |\na\ffi|^{4}\Big)       }
 {{\rm e}^{\ffi}}
\,\,\rmd\mu  \, .\nonumber
\end{align}
Concerning the boundary integral, it is convenient to split it into several pieces, writing
\begin{align*}
&\int\limits_{\pa(E_s^S\setminus U_\ep)}
\!\!\!\!\!\!\!
\big\langle X
\,\big|\,{\rm n}\big\rangle\,\,\rmd\sigma  \,\, 
=\!\!\!\!\!\!\!\!\!\!\!\!&
 \int\limits_{\{\ffi=S \} \setminus U_\ep}\!\!\!\!\!\!\!
\frac{   |\na \ffi|^{\beta+1}}{ {\rm e}^{S} }
\,\,\rmd\sigma \,\,  - \!\!\!\!\!\!\!
 \int\limits_{\{\ffi=s \} \setminus U_\ep }\!\!\!\!\!\!\!
\frac{   |\na \ffi|^{\beta+1}}{ {\rm e}^{s} }
\,\,\rmd\sigma  \,+\!\!\!\!\!\!
\int\limits_
{\pa U_\ep\cap E_s^S}
\!\!\!\!\!\!
\frac{   |\na \ffi|^{\beta+1}}{ {\rm e}^{\ffi} }\, \left\langle \frac{\na \ffi}{|\na\ffi|} \,\,\Bigg|\,\,  \frac{\na|\na\ffi|^2}{\big|\na|\na\ffi|^2 \big|}\right\rangle \,\,\rmd\sigma \,.
\end{align*}
Using the Dominated Convergence Theorem it is not hard to argue that the first two terms converge to the left hand side of~\eqref{eq:id_byparts_bis}, as $\ep \to 0$. To treat the last term, we observe that
\begin{equation*}
\left|\,\, \int\limits_
{\pa U_\ep\cap E_s^S}
\!\!\!\!\!\!
\frac{   |\na \ffi|^{\beta+1}}{ {\rm e}^{\ffi} }\, \left\langle \frac{\na \ffi}{|\na\ffi|} \,\,\Bigg|\,\,  \frac{\na|\na\ffi|^2}{\big|\na|\na\ffi|^2 \big|}\right\rangle \,\,\rmd\sigma \,\, \right| \,\, \leq \,\,\, \ep^{(\beta+1)/2} \!\!\!\! \int\limits_
{\pa U_\ep\cap E_s^S}
\!\!\!\!\!\!
{\rm e}^{-\ffi}\,\,\rmd\sigma \, .
\end{equation*}
Hence, it tends to $0$, as $\ep \to 0$, providing the desired identity.
\end{proof}
\begin{remark}
It is interesting to observe that the above proof does not use the fact that the Hausdorff dimension of ${\rm Crit}(\ffi)$ is bounded above by $(n-2)$. In fact, if $\overline{s}$ is a critical value for $\ffi$, one has that
\begin{equation*}
\int\limits_{\{\ffi=\overline{s}\} \setminus U_\ep}\!\!\!\!\!\!\!
{   |\na \ffi|_\g^{\beta+1}}
\,\,\rmd\sigma_\g \,\,\,\longrightarrow \!\!\!\!\!\! \int\limits_{\{\ffi=\overline{s} \} \setminus {\rm Crit}(\ffi)}\!\!\!\!\!\!\!\!\!\!\!\!
{   |\na \ffi|_\g^{\beta+1}}
\,\,\rmd\sigma_\g  \,\,\,= \!\! \int\limits_{\{\ffi=\overline{s} \} }\!\!\!
{   |\na \ffi|_\g^{\beta+1}}
\,\,\rmd\sigma_\g \, , \qquad \hbox{as $\ep \to 0$ \, ,}
\end{equation*}
since by definition $|\na \ffi|_g = 0$ on the critical set.
\end{remark}

As an immediate consequence of the {\em first integral identity}, we obtain a representation formula for the functions $\Phi_\beta$'s, 
as well as the continuity of 
$s \mapsto \Phi_\beta(s)$, for every $\beta \geq 0$.
This is stated in the following corollaries.

\begin{corollary}[Theorem~\ref{thm:main_conf}-(i) -- Representation Formula for $\Phi_\beta$]
\label{cor:repfi}
Let $(M,g,\ffi)$ be a solution to 
problem~\eqref{eq:pb_reform} and assume that  the growth condition $|\na\ffi|_g^2 \, = \, o\, ({\rm e}^\ffi)$ is satisfied, as $x \to \infty$. Then, for every 
$\beta\geq0$ and every $s\geq0 $, 
we have
\begin{equation}
\label{eq:fip_div}
\Phi_\beta(s)\,\,
= \,\,\,{\rm e}^s \!\!\!\!\int\limits_{\{\ffi>s\}}
\!\!
\frac{  |\na \ffi|_g^{\beta-2}  
\Big(\,|\na\ffi|_\g^4 -  \beta\,  \nana \ffi (\na\ffi, \!\na\ffi)   \Big) }{{\rm e}^\ffi}
\,\,\rmd\mu_\g\,.
\end{equation}
\end{corollary}
\begin{proof}
It is sufficient to apply the integral identity~\eqref{eq:id_byparts_bis} and observe that 
$$
\lim_{S \to +\infty} {{\rm e}^{-S}}\!\!\int\limits_{\{\ffi=S\}} \!\!\! 
|\na \ffi|_g^{\beta+1} 
\,\,\rmd \sigma_g \, = \, \lim_{S \to +\infty} {{\rm e}^{-S}}\, \Phi_\beta(S) \, = \, 0 \,,
$$
in virtue of Lemma~\ref{unif_bound}.
\end{proof}
\begin{corollary}[Theorem~\ref{thm:main_conf}-(i) -- Continuity]
\label{cor:contfi}
Let $(M,g,\ffi)$ be a solution to 
problem~\eqref{eq:pb_reform} and assume that  the growth condition $|\na\ffi|_g^2 \, = \, o\, ({\rm e}^\ffi)$ is satisfied, as $x \to \infty$. Then, for every $\beta\geq0$ and every $s\geq0 $, 
we have that the function $s \mapsto \Phi_\beta(s)$ defined in~\eqref{eq:fip} is continuous. 
\end{corollary}
\begin{proof}
As already observed, the quantity 
\begin{equation*}
\frac{  |\na \ffi|^{\beta-2}  
\Big(\beta\,  \nana \ffi (\na\ffi, \!\na\ffi) \, 
- \, |\na\ffi|^{4}  \Big) }{{\rm e}^{\ffi}}
\end{equation*}
is locally bounded for every $\beta \geq 0$. In particular it is locally summable. Hence, by the absolute continuity of the integral, we have that the right hand side of~\eqref{eq:id_byparts_bis} tends to $0$ when either $s \to S$ or $S \to s$. This implies the continuity of the assignment $s \mapsto {\rm e}^{-s} \Phi_\beta(s)$. The continuity of $s \mapsto \Phi_\beta(s)$ follows at once.
\end{proof}


\subsection{Second integral identity and proof of Theorem~\ref{thm:main_conf}-(ii).}
\label{sub:second}

We are now ready to prove the second of our integral identities. This formula represents the core of our analysis as it will be used to prove the differentiability of the functions $\Phi_\beta$'s  
together with the monotonicity and the rigidity statement claimed in~Theorem~\ref{thm:main_conf}-(ii). For the sake of clearness, we present a version of the second integral identity for regular values of $\ffi$.
In fact, according to Remarks~\ref{rem:deretta}, \ref{rem:uno_u} and~\ref{rem:deretta_g}, if $\overline{s}$ is a critical value of $\ffi$, some attention must be payed to the precise definition of the term 
\begin{equation*}
\int\limits_{\{\ffi=\overline{s}\}} \!\!\!\!
   |\na \ffi|_\g^\beta\,\HHH_\g 
\,\,\rmd\sigma_\g \, ,
\end{equation*}
at least when $(n-2)/(n-1) \leq \beta <1$ and thus the integrand is not necessarily bounded {\em a priori}.
\begin{lemma}
[Second Integral Identity for Regular Values]
\label{lem:cyl_reg}
Let $(M,g,\ffi)$ be a solution to 
problem~\eqref{eq:pb_reform} and assume that  the growth condition $|\na\ffi|_g^2 \, = \, o\, ({\rm e}^\ffi)$ is satisfied, as $x \to \infty$.
Then, for every 
$\beta\geq(n-2)/(n-1)$ and every couple of regular values $0 \leq s < S$ of the function $\ffi$, 
we have 
\begin{equation}
\label{eq:id_byparts_reg}
\!\!\int\limits_{\{\ffi=s\}} \!\!\!\!
\frac{   |\na \ffi|_\g^\beta\,\HHH_\g }{ {\rm e}^{s} }
\,\,\rmd\sigma_\g  
\,  - \!\!\! \int\limits_{\{\ffi=S\}}\!\!\!\!
\frac{   |\na \ffi|_\g^\beta\,\HHH_\g }{ {\rm e}^S }
\,\,\rmd\sigma_\g 
\,\, =  \!\!\!\!\!\!
\int\limits_{\{s<\ffi <S\}} \!\!\!\!\!
\frac{  |\na \ffi|_\g^{\beta-2}  
\Big(   \,  \big|\nana \ffi\big|_\g^2
 +  \, (\beta-2) \, \big| \na |\na \ffi |_\g   \big|_\g^2   \, \Big)       }{{\rm e}^{\ffi}}
\,\,\rmd\mu_g \,
\,\geq\,0\,.
\end{equation}
\end{lemma}
Before proceeding with the proof, it is worth noticing that the sign of the right hand side of~\eqref{eq:id_byparts_reg} is guaranteed for every $\beta$ greater that the threshold value $(n-2)/(n-1)$ by the fact that $\ffi$ is harmonic, and thus the refined Kato inequality 
\begin{equation*}
\big|\nana \ffi\big|_\g^2
 -  \, \Big(\frac{n}{n-1} \Big) \, \big| \na |\na \ffi |_\g   \big|_\g^2 \, \geq \,\, 0 \, 
\end{equation*}
is in force almost everywhere on $\{ s< \ffi < S\}$.

\begin{proof} 
As in the proof of Proposition \ref{prop:byparts},
we drop the subscript $\g$ to simplify the notation. For every $\beta\geq(n-2)/(n-1)$, 
we consider the vector field
$Y$ defined as
\begin{equation}
\label{def:campo_Y}
Y  
\,=\,  
\frac{\na|\na\ffi|^\beta}{{\rm e}^{\ffi}} \, .
\end{equation}
Wherever $|\na\ffi| > 0$, we compute the divergence of $Y$ with the help of equation~\eqref{eq:Bochner_pot_beta}, obtaining
\begin{equation}
\label{eq:id_div}
{\rm div}Y
\, = \,  
\, \beta\,\,  \frac{ |\na \ffi|^{\beta-2}  
\Big(   \,  \big|\nana \ffi\big|^2
 +  \, (\beta-2) \, \big| \na |\na \ffi |   \big|^2   \, \Big) }
 {{\rm e}^{\ffi}} \,\, \geq \,\, 0 \,\, .
\end{equation} 
For $0\leq s <S$, we set $E_s^S= \{ s < \ffi < S\}$ as in~\eqref{set_E}, and observe that if the closure of $E_s^S$ does not contain any critical point of $\ffi$, then the validity of~\eqref{eq:id_byparts_reg} follows directly from the Divergence Theorem.
As in the proof of Proposition~\ref{prop:byparts},
we now suppose without loss of generality that
there is only one critical value $\overline s$ of $\ffi$ in $(s,S)$. 
To deal with the presence of critical points inside $E_s^S$, we consider for every $\ep>0$ a smooth nondecreasing cut-off function $\chi_\ep: [0, + \infty) \longrightarrow [0,1]$ such that 
\begin{eqnarray*}
\chi_\ep(t) = 0 \, , &\quad & \hbox{for $t\leq \frac{1}{2} \ep$} \\
0 < \dot\chi_\ep(t) < 2 \ep^{-1} \, , &\quad & \hbox{for $\frac{1}{2} \ep < t < \frac{3}{2} \ep$}\\
\chi_\ep(t) = 1 \, , &\quad & \hbox{for $\frac{3}{2} \ep \leq t$} \,. 
\end{eqnarray*}
Using $\chi_\ep$, we define the smooth function $\Xi_\ep: M \longrightarrow [0,1]$ as
\begin{equation*}
\Xi_\ep (x) \, = \, \big( \, \chi_\ep \circ |\na \ffi|^2 \big) (x) \, . 
\end{equation*}
In particular, with the notation introduced in~\eqref{eq:pigiama}, we have that ${\rm supp} (\na \Xi_\ep) \subset \overline{U_{3\ep/2}} \setminus U_{\ep/2}$, since by the chain rule it holds
\begin{equation*}
\na \Xi_\ep (x) \, = \, \big( \, \dot\chi_\ep \circ |\na \ffi|^2 \big) (x)  \,\, \na |\na \ffi|^2 (x)\, .
\end{equation*} 
Taking advantage of our cut-off functions, we now apply the classical Divergence Theorem to the smooth vector field $\Xi_\ep \, Y$, obtaining 
\begin{align*}
&\int\limits_{\{\ffi=s\}} \!\!\!\!
\frac{   |\na \ffi|^\beta\,\HHH }{ {\rm e}^{s} }
\,\,\rmd\sigma  
\,  - \!\!\! \int\limits_{\{\ffi=S\}}\!\!\!\!
\frac{   |\na \ffi|^\beta\,\HHH }{ {\rm e}^S }
\,\,\rmd\sigma
\,\,\, =  \,\,
\int\limits_{E_s^S}\frac{\Xi_\ep \, {\rm div} Y}{\beta} \, \rmd\mu \,\, + \int\limits_{E_s^S}\frac{ \langle \na\Xi_\ep \, | \,   Y  \rangle}{\beta} \, \rmd\mu  \,\, = \\
&= \int\limits_{E_s^S} 
\Xi_\ep \, \frac{  |\na \ffi|^{\beta-2}  
\Big(   \,  \big|\nana \ffi\big|^2
 +  \, (\beta-2) \, \big| \na |\na \ffi |   \big|^2   \, \Big)       }{{\rm e}^{\ffi}}
\,\,\rmd\mu \,\, + \!\!\!\!\!\int\limits_{\overline{U_{\frac{3\ep}2}} \, \setminus \, U_{\frac{\ep}{2}}} \!\!\!\!\!\! 
\dot\chi_\ep (|\na \ffi|^2) \, \frac{  |\na \ffi|^{\beta-2}  
  \, \big| \na |\na \ffi |^2  \big|^2 }{2 {\rm e}^{\ffi}}
\,\,\rmd\mu \, .
\end{align*}
Now, it is clear that when $\ep \to 0$ the first term in the last row tends to the right hand side of~\eqref{lem:cyl_reg}, by the Monotone Convergence Theorem. We claim that for every $\beta>(n-2)/(n-1)$ the last term in the second row tends to $0$, as $\ep \to 0$, or equivalently 
\begin{equation}
\label{eq:claim}
\lim_{\ep \to 0} \,\,  \int\limits_{E_s^S}\frac{ \langle \na\Xi_\ep \, | \,   Y  \rangle}{\beta} \, \rmd\mu  \,\, = \,\, 0 \, .
\end{equation}
To prove such a claim, we use the Coarea Formula to write
\begin{align*}
\int\limits_{\overline{U_{\frac{3\ep}2}} \, \setminus \, U_{\frac{\ep}{2}}} \!\!\!\!\!\! 
\dot\chi_\ep (|\na \ffi|^2) \, \frac{  |\na \ffi|^{\beta-2}  
  \, \big| \na |\na \ffi |^2  \big|^2          }{2 {\rm e}^{\ffi}}
\,\,\rmd\mu \,\, = & \,\,\frac12 \int\limits_{\ep/2} ^{3\ep/2} \! \dot\chi_\ep(s) \,\,  s^{(\beta -2)/2} \,\, \rmd s \!\!\!\!\!\!\! \int\limits_{\{|\na \ffi|^2 = s\}} \!\!\!\!\!\! \frac{\big| \na |\na \ffi |^2   \big| }{ {\rm e}^\ffi} \,\,\rmd \sigma  \, \leq \\
\leq & \,\, \frac1\ep \int\limits_{\ep/2} ^{3\ep/2} \!   s^{(\beta -2)/2} \,\, \rmd s \!\!\!\!\!\!\! \int\limits_{\{|\na \ffi|^2 = s\}} \!\!\!\!\!\! \frac{\big| \na |\na \ffi |^2   \big| }{ {\rm e}^\ffi} \,\,\rmd \sigma \, ,
\end{align*}
where the last estimate follows by the structural properties of $\chi_\ep$. Since the integrand in the last term is a continuous function of $s$, the Mean Value Theorem for integrals insures the existence of a real number $r \in (\ep/2, 3\ep/2)$ such that 
\begin{equation*}
\frac1\ep \int\limits_{\ep/2} ^{3\ep/2} \!   s^{(\beta -2)/2} \,\, \rmd s \!\!\!\!\!\!\! \int\limits_{\{|\na \ffi|^2 = s\}} \!\!\!\!\!\! \frac{\big| \na |\na \ffi |^2   \big| }{ {\rm e}^\ffi} \,\,\rmd \sigma  \,\,\, = \,\,\, r^{(\beta -2)/2} \!\!\!\!\int\limits_{\{|\na \ffi|^2 = r\}} \!\!\!\!\!\! \frac{\big| \na |\na \ffi |^2  \big| }{ {\rm e}^\ffi} \,\,\rmd \sigma \, .
\end{equation*}
If we set 
\begin{equation*}
F(r) \,\,\, = \!\!\! \int\limits_{\{|\na \ffi|^2 = r\}} \!\!\!\!\!\! \frac{\big| \na |\na \ffi |^2   \big| }{ {\rm e}^\ffi} \,\,\rmd \sigma \,, 
\end{equation*}
it is sufficient to prove that $ \lim_{r \to 0}r^{(\beta-2)/2} F(r) = 0$ in order to obtain claim~\eqref{eq:claim} for every $\beta >(n-2)/(n-1)$. To accomplish this program, we first observe that, if ${\rm n}$ is the outer unit normal vector field to the set $U_r$ defined in~\eqref{eq:pigiama}, one can write
\begin{align*}
F(r) \,\,& = \, \int\limits_{\pa U_r} \frac{ \big\langle \na |\na \ffi|^2 \, \big| \,{\rm n}\big\rangle}{{\rm e}^\ffi} \, \rmd \sigma \,\, = \, \int\limits_{U_r}  {\rm div} \!\left(\frac{ \na |\na \ffi|^2 }{{\rm e}^\ffi} \right)  \rmd \mu \,\, = \,\, 2 \int\limits_{U_r} \, \frac{ |\nana \ffi|^2 }{{\rm e}^\ffi}  \, \rmd \mu \,\, = \\  &= \,\, 2 \int\limits_{0}^r \rmd s \!\!\!\!\!\int\limits_{\{ |\na \ffi|^2 = s \}} \!\!\!\!\!\!\frac{ |\nana \ffi|^2 }{{\rm e}^\ffi  \,  \big| \na |\na \ffi |^2   \big|    } \,\, \rmd \sigma \, ,
\end{align*}
where in the last two identities we have used formula~\eqref{eq:Bochner_pot} and the Coarea Formula, respectively. Differentiating the above expression with the help of the Fundamental Theorem of Calculus, and using the refined Kato inequality, we get the differential inequality
\begin{align*}
F'(r) \,\, &= \,\,\,  2 \!\!\!\!\!\! \int\limits_{\{ |\na \ffi|^2 = r \}} \!\!\!\!\!\!\frac{ |\nana \ffi|^2 }{{\rm e}^\ffi  \,  \big| \na |\na \ffi |^2   \big|    } \,\, \rmd \sigma  \,\, \geq \,\, 2 \,  \Big( \frac{n}{n-1}\Big)\!\!\!\!\! \int\limits_{\{ |\na \ffi|^2 = r \}} \!\!\!\!\!\!\! \frac{ \big|\na |\na \ffi|\big|^2 }{{\rm e}^\ffi  \,  \big| \na |\na \ffi |^2   \big|    } \,\, \rmd \sigma \,\, = \,\, \frac12 \, \,  \Big( \frac{n}{n-1}\Big) \, \frac{F(r)}{r} \, .
\end{align*}
Integrating this inequality between $r$ and fixed value $R>r$, 
one gets
\begin{equation*}
\frac{F(r)}{r^{\frac12  ( \frac{n}{n-1}) }} \,\, \leq \,\, \frac{F(R)}{R^{\frac12  ( \frac{n}{n-1}) }} \, .
\end{equation*}
In particular, it follows that for every $\beta > (n-2)/(n-1)$
\begin{equation*}
0 \,\, \leq \,\, r^{(\beta-2)/2} F(r) \,\, \leq \,\, \frac{F(R)}{R^{\frac12  ( \frac{n}{n-1}) }} \,\,  r^{\frac12  ( \beta - \frac{n-2}{n-1})} \, \longrightarrow \, 0 \,,  \qquad \hbox{as $r \to 0$} \, .
\end{equation*}
This completes the proof of claim~\eqref{eq:claim}, and in turns of the identity~\eqref{eq:id_byparts_reg} for every $\beta$ which is strictly above the optimal threshold. To obtain the desired identity also for $\beta = (n-2)/(n-1)$, it is now sufficient to pass to the limit in $\beta$, using the Dominated Convergence Theorem on the left hand side and the Monotone Convergence Theorem on the right hand side.
\end{proof}
As a direct application of the Second Integral Identity, we are going to show in the next Corollary that for every $\beta \geq (n-2)/(n-1)$ the function $\Phi_\beta$ is differentiable.

\begin{corollary}[Theorem~\ref{thm:main_conf}-(ii) -- Differentiability \& Monotonicity]
\label{cor:diffi}
Let $(M,g,\ffi)$ be a solution to 
problem~\eqref{eq:pb_reform} and assume that  the growth condition $|\na\ffi|_g^2 \, = \, o\, ({\rm e}^\ffi)$ is satisfied, as $x \to \infty$. Then, for every $\beta\geq (n-2)/(n-1)$ and every $s\geq0 $, 
we have that the function $s \mapsto \Phi_\beta(s)$ defined in~\eqref{eq:fip} is continuously differentiable. Moreover, the derivative $\Phi_\beta'$ is everywhere nonpositive and satisfies
\begin{equation}
\label{eq:premono}
\Phi_\beta' (S) \,  {\rm e}^{-S} \, 
-\,\,
\Phi_\beta' (s) \, {\rm e}^{-s}
\,= \,\,\,\beta \!\!\!\!\!\! \!
\int\limits_{\{s<\ffi<S\}} \!\!\!\!\!\!
\frac{  |\na \ffi|_g^{\beta-2}  
\Big(   \,  \big|\nana \ffi\big|_g^2
 +  \, (\beta-2) \, \big| \na |\na \ffi |_g\big|_g^2   
 \, \Big) }
 {{\rm e}^\ffi}
\,\, \, \rmd\mu_\g \,,
\end{equation}
for every $0 \leq s < S$.
\end{corollary}
\begin{proof}
We start showing that for every $\beta\geq(n-2)/(n-1)$ the function
\begin{equation*}
[0, + \infty) \setminus \ffi ({\rm Crit}(\ffi)) \, \ni \, s \, \longmapsto \,\Psi_\beta (s)\, =  \!\!\! \int\limits_{\{\ffi=s\}} \!\!\!\!
\frac{   |\na \ffi|^\beta\,\HHH }{ {\rm e}^{s} }
\,\,\rmd\sigma 
\end{equation*}
admits a (unique) continuous extension to the whole range of $\ffi$. Such an extension will be necessarily monotone, as it will satisfy the identity 
\begin{equation*}
\Psi_\beta(S_0) \, - \, \Psi_\beta(S)
\,\, =  \!\!\!\!\!\!
\int\limits_{\{S_0<\ffi <S\}} \!\!\!\!\!
\frac{  |\na \ffi|^{\beta-2}  
\Big(   \,  \big|\nana \ffi\big|^2
 +  \, (\beta-2) \, \big| \na |\na \ffi |   \big|^2   \, \Big)       }{{\rm e}^{\ffi}}
\,\,\rmd\mu \,
\,\geq\,0\,,
\end{equation*}
for every $0 \leq S_0 < S$, due to Lemma~\ref{lem:cyl_reg}. What we need to prove is that if $\overline{s}$ is a critical value of $\ffi$, then the formula
\begin{equation*}
\Psi_\beta(\overline{s}) \,\, = \,\, \lim_{s \to \overline{s}} \Psi_\beta({s})
\end{equation*}
yields a good definition for $\Psi_\beta(\overline{s})$. In other words, we have to show that the above limit exists and is finite. Since the singular values of $\ffi$ are discrete (see~\cite{Sou_Sou}), we let $\eta>0$ be such that the only regular value in $[\overline{s} - \eta, \overline{s}+ \eta]$ is given by $\overline{s}$. Using identity~\eqref{eq:id_byparts_reg}, it is easy to see that the assignment
\begin{align*}
[\overline{s} - \eta, \overline{s}) \, \ni \, s \, \, \longmapsto \,\, \Psi_\beta(s) \,\,= &\,\,\,\, \Psi_\beta(\overline{s} - \eta) \,\,\, - \!\!\!\!\!\!\int\limits_{\{\overline{s} - \eta <\ffi < s\}} \!\!\!\!\!\!\!
\frac{  |\na \ffi|^{\beta-2}  
\Big(   \,  \big|\nana \ffi\big|^2
 +  \, (\beta-2) \, \big| \na |\na \ffi |   \big|^2   \, \Big)       }{{\rm e}^{\ffi}}
\,\,\rmd\mu 
\end{align*}
is nonincreasing for every $\beta \geq (n-2)/(n-1)$. Moreover, by the same identity, it is immediate to deduce that it is bounded from below. In fact, one has that
\begin{equation*}
\Psi_\beta (s) \,\, \geq  \,\,\, \Psi_\beta(\overline{s} - \eta) \,\,\, - \!\!\!\!\!\!\!\!\!\!\int\limits_{\{\overline{s} - \eta <\ffi < \overline{s} + \eta\}} \!\!\!\!\!\!\!\!\!\!
\frac{  |\na \ffi|^{\beta-2}  
\Big(   \,  \big|\nana \ffi\big|^2
 +  \, (\beta-2) \, \big| \na |\na \ffi |   \big|^2   \, \Big)       }{{\rm e}^{\ffi}}
\,\,\rmd\mu \,\,
=  \,\,\, \Psi_\beta(\overline{s} + \eta) \,. 
\end{equation*}
This proves that $\lim_{s \to \overline{s}^-} \Psi_\beta(s)$ exists and is finite. Reasoning in the same way, one can prove that also $\lim_{s \to \overline{s}^+} \Psi_\beta(s)$ exists and is finite. Hence, it remains to show that the two limits coincide, but this follows directly from the absolute continuity of the integral.
In fact, an immediate consequence of Lemma~\ref{lem:cyl_reg} is that 
\begin{equation*}
\frac{  |\na \ffi|^{\beta-2}  
\Big(   \,  \big|\nana \ffi\big|^2
 +  \, (\beta-2) \, \big| \na |\na \ffi |   \big|^2   \, \Big)       }{{\rm e}^{\ffi}} \,\, \in \,\, L^1 \, \big(  \{ \, \overline{s} - \eta <\ffi < \overline{s} + \eta \, \}  , \mu \big) \, .
\end{equation*}
Hence, for every $\ep>0$ there exists a $\delta>0$ such that if $E$ is a measurable set with $\mu (E) \leq \delta$, then
\begin{equation*}
\int\limits_{E} \,\,  \frac{  |\na \ffi|^{\beta-2}  
\Big(   \,  \big|\nana \ffi\big|^2
 +  \, (\beta-2) \, \big| \na |\na \ffi |   \big|^2   \, \Big)       }{{\rm e}^{\ffi}}  \,\, \rmd \mu \,\, \leq \,\, \ep \, .\end{equation*}
To conclude it is thus enough to observe that for sufficiently small $t>0$, one has that $\mu \big( \{  \overline{s} - t < \ffi < \overline{s} +t      \} \big) < \delta$. 

We now pass to discuss, for every $\beta \geq (n-2)/(n-1)$, the differentiability of the auxiliary function $\Upsilon_\beta : [0, + \infty) \longrightarrow \R$, given by $\Upsilon_\beta (s) \, = \, {\rm e}^{-s} \Phi_\beta(s)$. By the First Integral Identity~\eqref{eq:id_byparts_bis} we have that
\begin{align*}
\frac{\Upsilon_\beta(s+h) - \Upsilon_\beta(s)}{h}
\,\, &= \,\, \frac{1}{h}\!\!\!\!\! \!\!\!
\int\limits_{\{s<\ffi<s+h\}} \!\!\!\!\!\!\!
\frac{  |\na \ffi|^{\beta-2}  
\Big(\, \beta\,  \nana \ffi (\na\ffi, \!\na\ffi)  \, - \, |\na\ffi|^{4} \,\Big) }
 {{\rm e}^\ffi}
\,\,\rmd\mu  \\
& = \,\, - \, \frac{1}{h} \,\int\limits_{s}^{s+h} \!\rmd t \!\!\!\int\limits_{\{ \ffi = t\}} \!\!\frac{|\na \ffi|^{\beta +1}}{{\rm e}^\ffi} \,\, \rmd \sigma \,\,  - \,\, \frac{\beta}{h}  \int\limits_{s}^{s+h} \!\rmd t \!\!\!\int\limits_{\{ \ffi = t\}} \!\!\frac{|\na \ffi|^{\beta}\, \HHH}{{\rm e}^\ffi} \,\, \rmd \sigma \\
& = \,\, - \, \frac{1}{h} \! \int\limits_{s}^{s+h} \!   \Upsilon_\beta(t) \,\,   \rmd t  \,\,  - \,\, \frac{\beta}{h} \! \int\limits_{s}^{s+h} \!  \Psi_\beta(t) \,\,\rmd t \, ,
\end{align*}  
Using the continuity of both $\Upsilon_\beta$ and $\Psi_\beta$ and invoking the Mean Value Theorem for integrals, we arrive at 
\begin{equation*}
\frac{\Upsilon_\beta(s+h) - \Upsilon_\beta(s)}{h}
\,\, = \,\, - \, \Upsilon_\beta (\xi_h)\, - \,\beta \, \Psi_\beta(\xi_h) \, ,
\end{equation*}
for some $\xi_h$ lying between $s$ and $s+h$. Letting $h \to 0$, we finally deduce that $s \mapsto \Upsilon_\beta(s)$ is a $\mathscr{C}^1$ function and that 
\begin{equation*}
\Upsilon_\beta' \,\, = \,\,  - \, \Upsilon_\beta \, - \, \beta \, \Psi_\beta\, .
\end{equation*}
This implies in turns that for every $\beta \geq (n-2)/(n-1)$ the function $s \mapsto \Phi_\beta(s) = {\rm e}^s \Upsilon_\beta(s)$ is differentiable. Moreover, the derivative coincides with $- \beta \,{\rm e}^s  \Psi_\beta$, so that it is continuous and satisfies the identity
\begin{equation*}
\label{basic_equiv}
\Phi_\beta' (S) \,  {\rm e}^{-S} \, 
-\,\,
\Phi_\beta' (S_0) \, {\rm e}^{-S_0}
\,= \,\,\,\beta \!\!\!\!\!\! \!
\int\limits_{\{S_0<\ffi<S\}} \!\!\!\!\!\!
\frac{  |\na \ffi|^{\beta-2}  
\Big(   \,  \big|\nana \ffi\big|^2
 +  \, (\beta-2) \, \big| \na |\na \ffi |\big|^2   
 \, \Big) }
 {{\rm e}^\ffi}
\,\, \, \rmd\mu \,.
\end{equation*}
for every $0 \leq S_0 < S$. Combining the latter identity with the upper bound for $\Phi_\beta$ obtained in Lemma~\ref{unif_bound} one easily deduces the monotonicity of $\Phi_\beta$ by the same argument outlined at the very end of Section~\ref{sec:conf_reform}.
\end{proof}

As a consequence of formula~\eqref{eq:premono} in the above corollary, we obtain a representation  formula for the derivatives of the functions $\Phi_\beta$'s, together with the rigidity statement claimed in Theorem~\ref{thm:main_conf}-(ii).

\begin{corollary}[Theorem~\ref{thm:main_conf}-(ii) -- Representation Formula for $\Phi_\beta'$ \& Rigidity Statement]
\label{cor:rep-rig}
Let $(M,g,\ffi)$ be a solution to 
problem~\eqref{eq:pb_reform} and assume that  the growth condition $|\na\ffi|_g^2 \, = \, o\, ({\rm e}^\ffi)$ is satisfied, as $x \to \infty$. Then, for every 
$\beta\geq(n-2)/(n-1)$ and every $s\geq0 $, 
we have
\begin{equation}
\label{eq:fip_div}
\Phi_\beta'(s)\,
\,= \,\, - \, \beta \,\, {\rm e}^s\!\!\!\!
\int\limits_{\{\ffi>s\}} \!\!
\frac{  |\na \ffi|_g^{\beta-2}  
\Big(   \,  \big|\nana \ffi\big|_g^2
 +  \, (\beta-2) \, \big| \na |\na \ffi |_g\big|_g^2   
 \, \Big) }
 {{\rm e}^\ffi}
\,\, \, \rmd\mu_g \,.
\end{equation}
Moreover, if there exists $s_0\geq0$ such that $\Phi_\beta'(s_0) = 0$ for some $\beta_0\geq(n-2)/(n-1)$,
then the manifold $(\{ \ffi \geq s_0\} , g)$ is
isometric to 
$\big(\,[s_0,+\infty)\times\{\ffi=s_0\},
d\varrho \otimes d\varrho + \g_{|\{ \ffi = s_0 \}}\big)$,
where $\varrho$ is the distance to $\{ \ffi = s_0\}$,
and $\ffi$ is an affine function of $\varrho$.
\end{corollary}
\begin{proof}
In virtue of identity~\eqref{eq:premono}, to prove that the representation formula~\eqref{eq:fip_div} is also satisfied, it is sufficient to find a sequence of positive real numbers $(s_n)_{n\in \mathbb{N}}$ tending to $+ \infty$ such that $\Phi_\beta'(s_n) \to 0$, as $n \to +\infty$. As a consequence of Corollary~\ref{cor:diffi} we have that, for every $\beta$ in the admissible range, the function $s \mapsto \Phi_\beta(s)$ is nonincreasing. On the other such function is also bounded from below and so admits a limit as $s \to + \infty$. In particular, for every $n \in \mathbb{N}$, there exists $k = k(n) \in \mathbb{N}$ such that $0 \leq \Phi_\beta (k) - \Phi_\beta (k+1) \leq 1/n$. By Lagrange's Theorem, there exists $s_n \in [k, k+1]$, such that $|\Phi_\beta'(s_n)| \leq 1/n$.

To prove the rigidity statement, observe
that if $\Phi_\beta'(s_0) = 0$ for some $s_0\geq0$ and some
$\beta\geq(n-2)/(n-1)$, 
then
\begin{equation*}  
|\nana \ffi|^2
 +  \, (\beta-2) \, | \na |\na \ffi |  |^2 \, \equiv \, 0 \, ,
\end{equation*}
in $\{ \ffi \geq s_0\}$,
due to the refined Kato inequality for harmonic functions.
Let us now distinguish two cases, depending if 
either $\beta>(n-2)/(n-1)$ or else $\beta=(n-2)/(n-1)$. In the former case it is immediate to conclude that 
$|\na|\na\ffi||\equiv0$ in $\{ \ffi \geq s_0\}$.
In turn, $|\na\ffi|$ is a nonzero constant in that region,
due to the last two conditions in \eqref{eq:pb_reform}.
Therefore, $\na \ffi$ is a nontrivial parallel vector field and 
by~\cite[Theorem 4.1-(i)]{Ago_Maz_1} we deduce that the Riemannian manifold $(\{ \ffi \geq s_0 \}, \g)$ is isometric to 
the manifold $\{ \ffi = s_0 \} \times [s_0, + \infty)$ 
endowed with the product metric 
$d\varrho \otimes d\varrho + \g_{|\{ \ffi = s_0 \}}$,
where $\varrho$ represents the distance to $\{ \ffi = s_0\}$.
Moreover, from the proof of Theorem 4.1-(i) in the mentioned paper, 
one gets that the function $\ffi$ can be expressed as an affine function of $\varrho$ in $\{ \ffi \geq s_0 \}$, 
i.e.~$\ffi = s_0 + \varrho \, |\na\ffi|_\g $, 
where $|\na \ffi|_\g$ is a positive constant. Let us now consider the case where $\beta = (n-2)/(n-1)$, 
which gives  
\begin{equation}
\label{eq:ref_kato}
|\nana \ffi|^2 \, = \, \frac{n}{n-1} \, | \na |\na \ffi ||^2 \, ,
\end{equation}
in $\{ \ffi \geq s_0\}$. 
Following the proof of~\cite[Proposition 5.1]{Bou_Car}, 
it is possible to deduce that $|\na \ffi|$ 
is a (necessarily nonzero) constant along the level sets of $\ffi$ and thus that the metric $g$ has a warped product structure in this region, namely 
\begin{equation}
\label{eq:warped}
g \,\, = \,\, d\varrho \otimes d\varrho + \eta^2(\varrho) \,\, \g_{|\{ \ffi = s_0 \}} \, ,
\end{equation}
for some positive warping function $\eta = \eta (\varrho)$. Moreover $\ffi$, $\rho$ and the warping factor $\eta$ satisfy the relationship
\begin{equation*}
\ffi (p) \, = \, s_0 \, + \, \kappa \!\!\int\limits_{0}^{\varrho(p)} \!\!\! \frac{{\rm d} \tau}{\eta(\tau)^{n-1}} \, 
\end{equation*}
for every point $p \in \{ \ffi \geq s_0 \}$ and some $\kappa \geq 0$. In particular, $\ffi$ and $\varrho$ share the same level sets and, by formula~\eqref{eq:warped}, these are totally umbilic. In fact one has
\begin{equation*}
\chg_{ij} \, = \, \frac12 \frac{\pa \g_{ij}}{ \pa \varrho} \, = \, \frac{d \log \eta}{d \varrho} \, g_{ij} \, .
\end{equation*}
As a consequence, the mean curvature is constant along each level set of $\ffi$.
%
%
%
%
%
%
Applying formula~\eqref{derivata_di_Phi} in~\ref{rem:deretta_g}, 
to every level set $\{\ffi = s \}$ with $s \geq s_0$ and $\beta=(n-2)/(n-1)$, one gets  
\begin{equation*}
\HHH \! \! \int \limits_{\{ \ffi = s \}} \!\!\! |\na \ffi|^{\frac{n-2}{n-1}} \, \rmd \sigma  \, = \, 0 \,,
\end{equation*}
since in virtue of~\eqref{eq:ref_kato} the right hand side of~\eqref{eq:fip_div} is always zero. This implies in turn that all the level sets $\{\ffi = s \}$ with $s\geq s_0$ are minimal and thus totally geodesic. From $\HHH \equiv 0$ one can also deduce that $\langle \na |\na\ffi|^2 \, \big| \, \na \ffi\rangle \equiv 0$ in $\{ \ffi \geq s_0\}$. Hence, $|\na \ffi|$ is constant in $\{ \ffi \geq s_0 \}$, and thus the conclusion follows arguing as in the case where $\beta>(n-2)/(n-1)$.
\end{proof}

\begin{remark}
\label{rem:analytic}
Observe that
if a solution $(M,g,\ffi)$ to problem~\eqref{eq:pb_reform}
is analytic
- as in the case when $(M,g,\ffi)$ comes from
a solution $u$ to problem~\eqref{eq:pb} 
through~\eqref{eq:tilde_g}
and~\eqref{def:ffi} - 
the conclusion of the rigidity statement in 
the above Corollary~\ref{cor:rep-rig} (as well as in Theorem~\ref{thm:main_conf}-(ii)) are stronger.
More precisely, 
the isometry between $(\{ \ffi \geq s_0\} , g)$ 
and 
$\big(\,[s_0,+\infty){\times}\{\ffi=s_0\},
d\varrho \otimes d\varrho + \g_{|\{ \ffi = s_0 \}}\big)$
improves to an isometry
between the whole manifold $(M,\g)$ and
$\big(\,[0,+\infty){\times}\{\ffi=0\},
d\varrho \otimes d\varrho + \g_{|\{ \ffi = 0 \}}\big)$.
\end{remark}


\section{Consequences of the Monotonicity-Rigidity Theorem}
\label{consequences}


In this section we discuss some of the analytic and geometric consequences 
of our Monotonicity-Rigidity Theorem~\ref{thm:main}. The first group of results will be deduced only using the fact that
\begin{equation}
\label{eq:fbi}
0 \, \leq \, - {F_\beta'}(1)
\,\,=\,\,  \beta \,
\int\limits_{ \pa \Omega } \,  |\DDD u|^\beta
\left[\, \HHH-
\big(\tfrac{n-1}{n-2}\big)
\, {|\DDD u|}
\, \right] 
 \, \rmd \sigma \, ,
\end{equation}
for every $\beta \geq (n-2)/(n-1)$, whereas the geometric inequalities of the second group (Theorem~\ref{cor:will-def} and Theorem~\ref{cor:mink-def}) will make a substantial use of the fact that 
\begin{equation}
\label{eq:fbii}
\left[\, {{\rm Cap}(\Omega)} \, \right]^{\frac{n-2-\beta}{n-2}} 
\!(n-2)^{\beta+1}\, |\Sph^{n-1}| \,\, =  \lim_{\tau \to + \infty}F_\beta(\tau) \,\,\leq \,\, F_\beta(1) \, = \, \int\limits_{ \pa \Omega} 
|\D u|^{\beta+1} \, \rmd \sigma
\,,
\end{equation}
where the limit has already been discussed in~\eqref{eq:u_exp}-\eqref{eq:lim_fb}.


\subsection{Consequences at the boundary}
\label{sub:con_bound}

We begin with the following sharp inequality, which says that the $L^p$-norm of the normal derivative at $\pa \Omega$ is always bounded above by the $L^p$-norm of the mean curvature, provided $p \geq 2- 1/(n-1)$.

\begin{theorem}
\label{cor:Lp_u_normal}
Let $u$ be a solution to problem~\eqref{eq:pb}.
Then, for every $p \geq 2-1/(n-1)$ the inequality
\begin{equation}
\label{eq:Lp_normal}
  \left\|  \frac{ \, \pa u \, }{\pa \nu} \right\|_{L^p( \pa \Omega)}  \!\!\!\!\! \leq \,\, \Big( \frac{n-2}{n-1} \Big) \,\, \big\|  {\HHH}  \big\|_{L^p(\pa \Omega)} 
\phantom{\quad}
\end{equation}
holds true, where 
$\HHH$ is the mean curvature of $\pa\Omega$ and $\nu$ is the unit normal vector of $\pa\Omega$ pointing toward the interior of $\R^n\setminus\overline\Omega$. 
Moreover, the equality is fulfilled for some $p \geq 2-1/(n-1)$ 
if and only if $u$ is rotationally symmetric. 
\end{theorem}
\begin{proof}
Setting $p= \beta +1$ in~\eqref{eq:fbi} and rearranging the terms, one gets
\begin{equation*}
 \int\limits_{\pa \Omega}  {  \, |\D  u|^p} \, \rmd \sigma \,\,\, \leq  \,\,\, \Big(\frac{n-2}{n-1} \Big) \, \int\limits_{\pa \Omega}  |\D u|^{p-1} {\HHH} \, \rmd \sigma \, .
\end{equation*}
The thesis follows from the H\"older inequality. The rigidity statement is a consequence of the rigidity  of the equality case in inequality~\eqref{eq:fbi}, which follows in turns from the Rigidity statement in Theorem~\ref{thm:main}-(ii).
\end{proof}
To proceed, we observe that letting $p \to + \infty$ in the previous inequality, one also obtains that 
\begin{equation}
\label{eq:Linf_normal}
\phantom{\!\!\!\!\!\!} \max_{\pa \Omega} \left| \frac{\pa u}{ \pa \nu} \right| \,\, \leq \,\, \Big( \frac{n-2}{n-1} \Big) \,\, \max_{\pa \Omega}  \big| {\HHH}  \big| \, .
\end{equation}
Unfortunately the corresponding rigidity statement does not survive the passage to the limit. However, one can recover the validity of such a statement either invoking the results in~\cite{Bor_Mas_Maz}, or using the following sharp gradient estimate together with the subsequent corollary

\begin{theorem}
[Sharp gradient estimate \`a la Colding]
\label{thm:sharp_bound_u}
Let $u$ be a solution to problem~\eqref{eq:pb}.
Then, for every $x\in\R^n\setminus\Omega$,
the inequality
\begin{equation*}
{|{\rm D}u|}(x)
\,\leq\,
\Big(\!\max_{\pa\Omega}{|{\rm D}u|}\Big) \,\, 
u^{\frac{n-1}{n-2}}(x)
\end{equation*}
holds true. Moreover, the equality is achieved at some point in $\R^n\setminus\overline\Omega$
if and only if $u$ is rotationally symmetric.  
\end{theorem}
\begin{proof}
Using the expansion~\eqref{eq:u_exp} in combination with~\eqref{eq:tilde_g} and~\eqref{def:ffi}, it is immediate to deduce that $|\na \ffi|_g^2 = O(1)$, as $x \to \infty$.
Hence, the hypothesis of Proposition~\ref{thm:sharp_bound_fi} are largely satisfied. The thesis can be now easily deduced from~\eqref{gra_est}, with the help of~\eqref{eq:P_function}.
\end{proof}

From inequality \eqref{eq:Linf_normal} 
in Corollary~\ref{cor:Lp_u_normal} and from 
Theorem~\ref{thm:sharp_bound_u} above, it is immediate to deduce the following Corollary.

\begin{corollary}
Let $u$ be a solution to problem~\eqref{eq:pb}.
Then, for every $x\in\R^n\setminus\Omega$,
the inequality
\begin{equation*}
\qquad \quad {|{\rm D}u|}(x)
\,\leq\,
 \bigg(\frac{n-2}{n-1}\bigg) \,\Big(\!\max_{\pa\Omega}{|\HHH|}\Big) \,\, 
u^{\frac{n-1}{n-2}}(x)
\end{equation*}
holds true. Moreover, the equality is achieved at some point in $\R^n\setminus\overline\Omega$
if and only if $u$ is rotationally symmetric.
\end{corollary}
In particular, the equality case in~\eqref{eq:Linf_normal} is characterized by a geometric rigidity, as desired.

Recalling that the electrostatic capacity of a charged body $\Omega$ can be computed in terms of the exterior normal derivative as
\begin{equation*}
{\rm Cap} (\Omega) \,\, = \,\,  - \, \frac{1}{(n-2) |\Sph^{n-1}|}\int_{ \pa \Omega} \!\! \frac{ \, \pa u \, }{\pa \nu} \, \rmd \sigma \,,\phantom{\quad\quad}
\end{equation*}
and using the H\"older inequality, it is not hard to deduce from~\eqref{eq:Lp_normal} and~\eqref{eq:Linf_normal} the following geometric upper bounds for the capacity.
Observe that $\pa u/\pa\nu=-|\D u|<0$ on $\pa\Omega$,
due to the definition of $\nu$. 
Other geometric upper and lower bounds
for the capacity are obtained for example
in~\cite{Bra_Mia,Fre_Sch,Hur_Pal_Rit,Xiao}.

\begin{corollary}
\label{cor:capupper}
Let $\Omega \subset \mathbb{R}^n$, $n\geq3$,
be a bounded domain with smooth boundary. 
Then, for every $p\geq 2-1/(n-1)$, 
the inequality
\begin{equation}
\label{eq:cap_p_upper}
{\rm Cap} (\Omega) \,\, \leq \,\, 
\frac{|\pa \Omega|}{|\Sph^{n-1}|} 
\,\left(   \fint_{ \pa \Omega}  \left|\frac{\HHH}{n-1}\right|^{p}  \rmd \sigma \right)^{\!1/p}
\end{equation}
holds true, where $\HHH$ is the mean curvature of $\pa \Omega$. Moreover, the equality is fulfilled for some $p\geq 2-1/(n-1)$ 
if and only if $\Omega$ is a round ball. 
Finally, letting $p \to + \infty$ in the previous inequality, one has that 
\begin{equation}
\label{eq:capinf_upper}
{\rm Cap} (\Omega) \,\, \leq \,\, 
\frac{|\pa \Omega|}{|\Sph^{n-1}|} 
\,\, \max_{\pa \Omega}  \bigg| \frac{\HHH}{n-1}  \bigg| \, .
\phantom{\qquad\quad\,}
\end{equation}
Moreover, the equality is fulfilled if and only if $\Omega$ is a round ball.
\end{corollary}


\subsection{Global geometric consequences}
\label{sub:con_glob}


So far we have used the local feature of the monotonicity, namely inequality~\eqref{eq:fbi}, to prove a first group of corollaries of Theorem~\ref{thm:main}. 
To deduce further consequences of our main theorem, 
we now exploit the global feature of the monotonicity, comparing the behaviour of our quantities at the boundary and at large level sets of $u$, with the help of~\eqref{eq:fbii}. We start with the proof of the Weighted Minkowski Inequality.

\begin{proof}[proof of Theorem~\ref{cor:mink-def}]
Multiplying inequality~\eqref{eq:fbii} by $(n-1)/(n-2)$, we obtain 
\begin{equation*}
\left[\, {{\rm Cap}(\Omega)} \, \right]^{\frac{n-2-\beta}{n-2}} \!
(n-1)\,(n-2)^{\beta}\, |\Sph^{n-1}|  \,\,\leq \int\limits_{ \pa \Omega} 
|\D u|^{\beta}\,  \HHH \, \rmd \sigma\,\,  -\!\!
\int\limits_{ \pa \Omega } \,  |\DDD u|^\beta
\left[\, \HHH-
\big(\tfrac{n-1}{n-2}\big)
\, {|\DDD u|}
\, \right] 
 \, \rmd \sigma
\,,
\end{equation*}
where we have added and subtracted the quantity $\int_{ \pa \Omega} 
|\D u|^{\beta}\,  \HHH \, \rmd \sigma $ to the right hand side. A simple rearrangements of the terms leads to
\begin{equation*}
\int\limits_{ \pa \Omega } \,  \left|\frac{\DDD u}{n-2}\right|^\beta
\left[\, \frac{\HHH}{n-1} \,-
\, {\left|\frac{\DDD u}{n-2}\right|}
\, \right] 
 \, \rmd \sigma \,\, \leq \,\, \int\limits_{ \pa \Omega} \,
 \left|\frac{\DDD u}{n-2}\right|^\beta \!  \frac{\HHH}{n-1} \, \rmd \sigma  \, - \, \left[\, {{\rm Cap}(\Omega)} \, \right]^{\frac{n-2-\beta}{n-2}}\, |\Sph^{n-1}| 
\end{equation*}
Notice that the left hand side is nonnegative in view of Theorem~\ref{thm:main}-(ii).
Setting $\beta= (n-2)/(n-1)$, the last summand on the right hand side can be rewritten as
\begin{equation*}
\left[\, {{\rm Cap}(\Omega)} \, \right]^{\frac{n-1}{n-2}}\, |\Sph^{n-1}| \, = \, \left( \fint_{\pa \Omega }\left| \frac{\DDD u}{n-2} \right| \rmd \sigma\right)^{\!\! \frac{n-2}{n-1}} \! |\pa \Omega|^{\frac{n-2}{n-1}} \,\, |\Sph^{n-1}|^{\frac{1}{n-1}} \, .
\end{equation*}
Substituting this expression in the last inequality, we arrive with some algebraic manipulations at
\begin{equation*}
\int\limits_{ \pa \Omega } \!  
\left[\, \frac{\HHH}{n-1} \,-
\, {\left|\frac{\DDD u}{n-2}\right|}
\, \right] 
\left(\frac{|\DDD u|}{\fint_{\pa \Omega }\left| {\DDD u} \right| \rmd \sigma}\right)^{\!\!\frac{n-2}{n-1}}  \!\! \rmd \sigma  
\,\leq \,
\int\limits_{ \pa \Omega} 
\!   \frac{\HHH}{n-1}
\left(\frac{|\DDD u|}{\fint_{\pa \Omega }\left| {\DDD u} \right| \rmd \sigma}\right)^{\!\!\frac{n-2}{n-1}}  \!\! \rmd \sigma \,\,\, - \,\,\, |\pa \Omega|^{\frac{n-2}{n-1}} \, |\Sph^{n-1}|^{\frac{1}{n-1}}
\end{equation*}
Recalling the definition of the measure $\overline\sigma$, it is then easy to realise that the last inequality coincides with~\eqref{eq:minki}. The rigidity follows at once by the rigidity statement in Theorem~\ref{thm:main}-(ii).
\end{proof}
We are now ready to deduce the Quantitative Willmore-type Inequality.

\begin{proof}[proof of Theorem~\ref{cor:will-def}]
Multiplying the last inequality in the proof of Theorem~\ref{cor:mink-def} by the factor
\begin{equation*}
\left(\fint_{\pa \Omega }\left| {\DDD u} \right| \rmd \sigma\right)^{\!\!\frac{n-2}{n-1}} \, ,
\end{equation*}
and using~\eqref{formula_cap}, we obtain 
\begin{align*}
\int\limits_{ \pa \Omega }  
\left[\, \frac{\HHH}{n-1} \,-
\, {\left|\frac{\DDD u}{n-2}\right|}
\, \right] 
{|\DDD u|}^{\frac{n-2}{n-1}}  \,\, \rmd \sigma  
\,\, \leq \,\,
\int\limits_{ \pa \Omega} 
\!   \frac{\HHH}{n-1} \,\,  {|\DDD u|}^{\frac{n-2}{n-1}}
\,\,\rmd \sigma \,\, - \,\,  
|\Sph^{n-1}|^{{1}/({n-1})}
\left(\int_{\pa \Omega }\left| {\DDD u} \right| \rmd \sigma\right)^{\!\!\frac{n-2}{n-1}} \\
 \leq \,\, 
\left[   \left( \int_{ \pa \Omega}  \left|\frac{\HHH}{n-1}\right|^{n-1} \!\! \rmd \sigma \right)^{\!\!1/(n-1)} \!\!\!\!\!\!\!\!\!\!\!\! - \,\,
|\Sph^{n-1}|^{{1}/({n-1})} \right]
\left(\int_{\pa \Omega }\left| {\DDD u} \right| \rmd \sigma\right)^{\!\!\frac{n-2}{n-1}} ,\phantom{\qquad\qquad\qquad}
\end{align*}
where in the last inequality we used the H\"older Inequality. The thesis~\eqref{eq:will} follows now from simple algebraic manipulations, noticing that the leftmost hand side is nonnegative by Theorem~\ref{thm:main}-(ii). The rigidity also follows at once by the rigidity statement in Theorem~\ref{thm:main}-(ii).
\end{proof}


\subsection{Quantitative Willmore Inequality in $3$-D} 


We conclude with the observation that for $n=3$ one can provide a different quantitative version of the Willmore Inequality, beside the one that 
follows from the general statement of Theorem~\ref{cor:will-def}. For the ease of reference, we recall that setting $n=3$ in~\eqref{eq:will} one gets
\begin{equation}
\label{eq:will3}
\int\limits_{\pa \Omega} \sqrt{\frac{|\D u|}{4 \pi \, {\rm Cap} (\Omega)}} \,\,  \left( \HHH - 2 |\D u| \right) \, \rmd \sigma \,\,\, \leq \,\,\,\,   \left( \int_{ \pa \Omega} \!\! {\HHH}^2  \,\rmd \sigma \right)^{\!\!1/2} \! - \,\,\, \sqrt{16 \pi} \,,
\end{equation}
where the deficit on the left hand side is nonnegative and optimal, due to Theorem~\ref{thm:main}-(ii). Here optimal means that as soon as it is zero, the right hand side is also zero. 

A slightly different path leads to the following statement, in which the deficit is still optimal, although different.
\begin{corollary}
Let $\Omega \subseteq \R^3$
be a bounded domain with smooth boundary. Then, the inequality	
\begin{equation}
\label{eq:will3bis}
\int\limits_{\pa \Omega}   \left( \HHH - 2 |\D u| \right)^2 \, \rmd \sigma \,\,\, \leq \,\,\, \int_{ \pa \Omega} \!\! {\HHH}^2  \,\rmd \sigma  \,\,  - \,\,\, {16 \pi} \,,
\end{equation}
holds true, where $\HHH$ is the mean curvature of $\pa \Omega$ and $u$ is the capacitary potential due to the uniformly charged body $\Omega$.
Moreover, the deficit on the left hand side is optimal in the sense that if it vanishes, then the right hand side also vanishes and $\Omega$ is a round ball.
\end{corollary}
\begin{proof}
Notice that for $n=3$ and $\beta = 1$ the inequalities~\eqref{eq:fbi} and~\eqref{eq:fbii} give 
\begin{equation}
4\pi \,\, \leq  \, \int\limits_{\pa \Omega}    |\D u|^2 \, \rmd \sigma  \, \qquad \hbox{and} \qquad 0 \,\, \leq  \int\limits_{\pa \Omega}   |\D u | \, \left( \HHH - 2 |\D u| \right) \, \rmd \sigma \, .
\end{equation}
Starting from these results, it is possible to deduce at once the following chain of inequalities
\begin{align}
16 \pi \,\, \leq  & \,\,\,  4  \int\limits_{\pa \Omega}    |\D u|^2 \, \rmd \sigma \,\, \leq \,\,  2 \int\limits_{\pa \Omega}    |\D u| \, \HHH \, \rmd \sigma \, + \, 2 \int\limits_{\pa \Omega}   |\D u | \, \left( \HHH - 2 |\D u| \right) \, \rmd \sigma  \,\, = \\
= & \,\,\, 4  \int\limits_{\pa \Omega}   |\D u | \, \left( \HHH -  |\D u| \right) \, \rmd \sigma  \,\, = \,\, \int\limits_{\pa \Omega}    \HHH^2 \, \rmd \sigma \,\, -  \int\limits_{\pa \Omega}   \left( \HHH - 2 |\D u| \right)^2 \, \rmd \sigma \,. 
\end{align}
The thesis follows by a simple rearramgement. If the deficit vanishes, then $\HHH \equiv 2 |\D u|$ on $\pa \Omega$, and this triggers the Rigidity statement in Theorem~\ref{thm:main}-(ii).
\end{proof}


\subsection*{Acknowledgements}
\emph{The author are grateful to G. Crasta, A.~Farina, I. Fragal\`a, C. Mantegazza, J. Metzger, M. Novaga, and 
D.~Peralta-Salas for useful comments and discussions during the preparation of the paper.
The authors are members of Gruppo Nazionale per
l'Analisi Matematica, la Probabilit\`a e le loro Applicazioni (GNAMPA),
which is part of the Istituto Nazionale di Alta Matematica (INdAM),
and partially funded by the GNAMPA project
``Principi di fattorizzazione, formule di monotonia e disuguaglianze geometriche''.
}


\bibliographystyle{plain}

\end{document}